\newcommand{\References}{references}

\documentclass[a4paper]{article}

\usepackage[T1]{fontenc}

\usepackage{amsthm}%

\usepackage{graphicx}%
\usepackage{multirow}%
\usepackage{amsmath,amssymb,amsfonts}%
\usepackage{mathrsfs}%
\usepackage{xcolor}%
\usepackage{textcomp}%
\usepackage{manyfoot}%
\usepackage{booktabs}%
\usepackage{algorithm}%
\usepackage{algorithmicx}%
\usepackage{algpseudocode}%
\usepackage{listings}%
\usepackage[colorlinks=true]{hyperref}
\usepackage{cleveref}

\usepackage{dsfont}
\usepackage{tikz}
\usetikzlibrary{patterns,math}
\usetikzlibrary{arrows}
\usetikzlibrary{shapes}
\usetikzlibrary{decorations.pathmorphing}
\usetikzlibrary{decorations.pathreplacing}
\usetikzlibrary{decorations.shapes}
\usetikzlibrary{decorations.text}
\usetikzlibrary{patterns}
\usetikzlibrary{backgrounds}
\usepackage{tkz-euclide}
\usepackage{geometry}
\usepackage{pgfplots}
\pgfplotsset{compat=1.18} 

\usepackage{textcomp}   
\def\BibTeX{{\rm B\kern-.05em{\sc i\kern-.025em b}\kern-.08em
    T\kern-.1667em\lower.7ex\hbox{E}\kern-.125emX}}

\usepackage{mathtools}
\usepackage{enumitem}
{
\newtheorem{definition}{Definition}[section]
\theoremstyle{definition}
\newtheorem{remark}[definition]{Remark}
\theoremstyle{plain}
\newtheorem{proposition}[definition]{Proposition}
\newtheorem{theorem}[definition]{Theorem}
}

\makeatletter
\newcommand\RedeclareMathOperator{%
  \@ifstar{\def\rmo@s{m}\rmo@redeclare}{\def\rmo@s{o}\rmo@redeclare}%
}
\newcommand\rmo@redeclare[2]{%
  \begingroup \escapechar\m@ne\xdef\@gtempa{{\string#1}}\endgroup
  \expandafter\@ifundefined\@gtempa
     {\@latex@error{\noexpand#1undefined}\@ehc}%
     \relax
  \expandafter\rmo@declmathop\rmo@s{#1}{#2}}
\newcommand\rmo@declmathop[3]{%
  \DeclareRobustCommand{#2}{\qopname\newmcodes@#1{#3}}%
}
\@onlypreamble\RedeclareMathOperator
\makeatother

{\left(\begin{smallmatrix}}%
{\end{smallmatrix}\right)}%

\newenvironment{smallbmatrix}%
{\left[\begin{smallmatrix}}%
{\end{smallmatrix}\right]}%

\newcommand{\N}{\mathds{N}}
\newcommand{\R}{\mathds{R}}
\newcommand{\Rp}{\R_{\geq0}}
\newcommand{\Rpp}{\R_{>0}}

\newcommand{\dB}{\mathcal{B}}

\newcommand{\fa}{\ \forall \, }

\newcommand{\rbl}{\left (}
\newcommand{\rbr}{\right )}

\newcommand{\nl}{\left\|}
\newcommand{\nr}{\right\|}
\newcommand{\cbl}{\left\lbrace }
\newcommand{\cbr}{\right\rbrace }

\newcommand{\Norm}[2][ ]{\nl #2 \nr_{#1}}
\newcommand{\SNorm}[1]{\Norm[\infty]{#1}}

\newcommand{\setdef}[2]{\cbl\ #1\ \left|\, \vphantom{#1} #2\, \right.\cbr}

\newcommand{\Gl}{{\rm Gl}}

\newcommand{\cC}{\mathcal{C}}

\newcommand{\cK}{\mathcal{K}}

\newcommand{\cF}{\mathcal{F}}
\newcommand{\cG}{\mathcal{G}}

\newcommand{\ModMax}{\bar u, \bar \gamma}
\newcommand{\cM}{\mathcal{M}_{\ModMax}}

\newcommand{\cN}{\mathcal{N}}

\DeclareMathOperator*{\rf}{ref}
\DeclareMathOperator*{\esssup}{ess\,sup}

\renewcommand{\l}{\lambda}

\newcommand{\me}{\mathrm{e}}

\newcommand{\con}{\mathcal{C}}

\newcommand{\cL}{\mathcal{L}}

\RedeclareMathOperator*{\Im}{Im}
\RedeclareMathOperator*{\Re}{Re}
\renewcommand{\phi}{\varphi}

\newcommand{\vp}{\varphi}

\newcommand{\dd}[2][ ]{\tfrac{\text{\normalfont d}#1}{\text{\normalfont d}#2}}

\newcommand{\yM}{y_{\rm M}}
\newcommand{\yref}{y_{\rm ref}}
\newcommand{\uMPC}{u_{\rm FMPC}}
\newcommand{\umax}{\bar u}

\newtheorem{algo}{Algorithm}
\Crefname{definition}{Definition}{Definitions}
\Crefname{assumption}{Assumption}{Assumptions} 
\Crefname{lemma}{Lemma}{Lemmata}
\Crefname{remark}{Remark}{Remarks}
\Crefname{theorem}{Theorem}{Theorems}

\newcommand{\email}[1]{\protect\href{mailto:#1}{#1}}
\newenvironment{@abssec}[1]{%
     \if@twocolumn
       \section*{#1}%
     \else
       \vspace{.05in}\footnotesize
       \parindent .2in
         {\upshape\bfseries #1: }\ignorespaces 
     \fi}
     {\if@twocolumn\else\par\vspace{.1in}\fi}
\newenvironment{keywords}{\begin{@abssec}{Keywords}}{\end{@abssec}}
\begin{document}
\title{
Safe continual learning in model predictive control with prescribed bounds on the tracking error\footnote{We gratefully acknowledge funding by the German Research Foundation (DFG; project numbers 471539468 and 507037103).}
}

 \author{Lukas Lanza\thanks{{Institute of Mathematics}, {Technische Universit\"at Ilmenau}, {98693
     Ilmenau}, {Germany} (\email{lukas.lanza@tu-ilmenau.de}, \email{karl.worthmann@tu-ilmenau.de})}
    \and{Dario Dennst\"adt\thanks{{Institute of Mathematics}, {Universit\"at
    Paderborn}, {33098 Paderborn}, Germany (\email{dario.dennstaedt@uni-paderborn.de}, \email{thomas.berger@uni-paderborn.de})}}
    \and{Thomas Berger}\footnotemark[3]
    \and{Karl Worthmann}\footnotemark[2]
}

\maketitle
\begin{abstract}
    We develop a three-component Model Predictive Control (MPC) algorithm to achieve output-reference tracking with prescribed performance for continuous-time nonlinear systems. 
    One component is so-called funnel MPC, which achieves reference tracking with prescribed performance for the model output for suitable models.
    Recently, this MPC algorithm has been combined with a model-free reactive feedback controller (second component) to account for model-plant mismatches, bounded disturbances, and uncertainties.
    By construction, this two-component controller defines a robust funnel MPC algorithm.
    It achieves output-reference tracking within prescribed bounds on the 
    tracking error for {a class of} unknown nonlinear systems. 
    In this paper, we extend the robust funnel MPC by a machine learning component
    to adapt the underlying model to the system data and, thus, improve the contribution of MPC.
    We derive sufficient structural conditions to define a class of models for funnel MPC, and provide a characterization of suitable learning schemes.
    Since robust funnel MPC is inherently robust and the evolution of the tracking error in the prescribed performance funnel is 
    guaranteed, the additional learning component is able to perform the learning task online --~even without an initial model or offline training.
\end{abstract}

\begin{keywords}
    Model predictive control, robust control, funnel control, safe learning
\end{keywords}

\section{Introduction}
Model Predictive Control (MPC) is a well-established control technique for constrained linear as well as for nonlinear multi-input multi-output systems.
Using a model of the system to be controlled, the idea is to iteratively solve Optimal Control Problems~(OCPs) based on predictions 
about the system %
behavior on a finite-time horizon, see e.g., the textbook~\cite{grune2017nonlinear}, and the references therein.
MPC is nowadays widely used {in} various applications, see e.g.,~\cite{BadgQin21} and the references therein.

Since MPC heavily relies on the availability and accuracy of the underlying model of the actual system, some research effort has been made to compensate for model-plant mismatches and external disturbances.
One research branch focuses on the MPC algorithm itself and its robustification, see e.g.,~\cite{BergDenn23,Buja21,KohlSolo20,SunDai19,RakoDai22}, and the references therein, respectively. 
For instance, in~\cite{KohlSolo20}, a robust MPC framework was developed for nonlinear discrete-time systems. The uncertainties and disturbances are compensated by introducing tightening tubes around the input and output constraints to ensure robust satisfaction of the constraints. 
A similar tube-based technique was used in~\cite{RakoDai22}.
In~\cite{BergDenn23}, it was shown that combining funnel MPC, which achieves reference tracking within prescribed error margins, cf.~\cite{BergDenn21}, with an additional robust feedback controller leads to a control scheme, which achieves a tracking objective even in case of a severe model-plant mismatch. 

Due to the recent successes in the field of machine learning, there have
been attempts to utilize such techniques, especially reinforcement
learning (RL), to learn an optimal control policy and mimic the behavior of
(robust) MPC algorithms, see e.g.,~\cite{amos2018differentiable,10606056}.
In~\cite{hassanpour2024practically}, RL
methods have been applied to the example of a chemical reactor utilizing
computations from an industrial implemented model predictive
control.
A different research branch, however, focuses on the idea to achieve robust constraint satisfaction of the actual system via adaption of the underlying model, see e.g.,~\cite{Aswa13,BerbKoeh20,LoreCann19} and the references therein.
For instance, in~\cite{BerbKoeh20}, the notion of persistently exciting data, cf.~\cite{Will05,faulwasser2022behavioral}, is exploited to update the model and show initial and recursive feasibility of the resulting MPC scheme.
In~\cite{hewing2019cautious,maiworm2021online}, combining MPC with Gaussian process-based learning schemes was 
proposed to achieve predictive control with stability guarantees.
In~\cite{maiworm2021online}, a NARX model was incorporated in the learning scheme.
The proposed controller in~\cite{hewing2019cautious} is formulated to address satisfaction of chance constraints, and its functioning is demonstrated with an autonomous racing vehicle.
A similar approach was used in~\cite{matschek2023safe} to perform safe learning-based control in robotics.
The issue of malfunctioning of the machine learning scheme is considered in~\cite{zieger2022non}, where the underlying neural network is forced to stay close to a predefined nominal model, and hence the MPC scheme safely achieves the control objective.
Predictive safety filters are an MPC variant
which is also closely related to tube-based approaches and
allows the application of learning-based control techniques while guaranteeing compliance with constraints.
The core idea is that the predictive safety filter verifies a control input signal proposed by a learning algorithm against a model.
If the proposed control signal is deemed safe it is applied to the system, otherwise
it is modified as little as necessary to guarantee constraint satisfaction, see~e.g.,~\cite{Wabersich18,Wabersich21,Wabersich23}.
A comprehensive overview of the application of various safe learning methods in MPC can be found in~\cite{HewingWaber20} and the references therein.
Combining a data-based controller with a reactive feedback controller to safeguard the transient behavior is presented in~\cite{gottschalk2024reinforcement} for policy iteration, in~\cite{schmitz2023safe} for data-enabled predictive control (DeePC~\cite{coulson2019data}) based on the fundamental lemma by Willems and coauthors~\cite{willems2005note} for LTI systems, in~\cite{lanza2024_sampleddata} for Q-learning, and in~\cite{BoldLanzWoth2024_Koopman} for nonlinear systems using Koopman-based MPC.

A particular control task, which is considered in the present article, is output-reference tracking {with prescribed performance}.
This differs from achieving asymptotic tracking, i.e., ensuring that the tracking error converges to zero,  which is 
is the primary focus of most MPC-based output tracking approaches, see e.g.,~\cite{kohler2018nonlinear,limon2018nonlinear}.
Several adaptive model-free control techniques like \emph{funnel control}~\cite{IlchRyan02b,BergIlch21} or \emph{prescribed performance
control}~\cite{bechlioulis2008robust,bechlioulis2014low} have been developed to achieve a prescribed transient behavior of the system output.
Since those approaches do not use a model of the system, the respective controllers cannot ``plan ahead'' and, thus, only react to the most-recent measurement.
This often results in a rapidly changing control signal with peaks and overall high control values.
Inspired by these techniques, continuous-time
funnel MPC was introduced in~\cite{berger2019learningbased} and further developed in~\cite{BergDenn21,BergDenn22,BergDenn23b}.
Incorporating a model to overcome the shortcomings of the adaptive control approaches,
funnel MPC utilizes a particular choice of the stage cost.
By penalizing the distance of the tracking error to the prescribed error boundary via a ``funnel-like'' stage cost,
the funnel MPC scheme achieves output tracking for systems with global strict relative degree one and bounded-input bounded-state stable internal dynamics, where the input and output are of the same dimension.
Within this framework it was shown that the tracking objective is achieved and funnel MPC is 
initially and recursively feasible, without 
incorporating terminal conditions, and arbitrary choices of the prediction horizon.
Although not proven yet mathematically, numerical results indicate that this
control algorithm can successfully be applied to systems with higher relative
degree, cf.~\cite{BergDenn21,berger2019learningbased}, as well as
\Cref{Sec:Simulation_RelDeg2} in the present article. Extensions to arbitrary
relative degree based on different stage cost functions are discussed
in~\cite{BergDenn22,BergDenn23b} and initial and recursive feasibility is
proved. 
Besides funnel MPC, other works also utilize MPC to prescribe the closed-loop behavior in the context of output-reference tracking.
For discrete-time systems, constraint satisfaction is ensured by assuming
suitable stabilizability and detectability conditions and a sufficiently long
prediction horizon in~\cite{kohler2021constrained}.
In~\cite{kheradmandi2017prescribing}, an MPC algorithm is used to achieve output
tracking within given boundaries for chemical plants in the form  of a numerical
case study. 

Many MPC schemes assume access to the full state of the system to be controlled; however, this is not satisfied in general.
Therefore, in~\cite{BergDenn23}, a distinction between the system to be controlled and the model to be used in the MPC algorithm is made,
and only access to the model's state and to output measurements (but no state measurements) of the real system is used to initialize the model.
In the present article we make use of that distinction, too, see \Cref{Sec:SystemModelClass}.
A different approach is pursed in~\cite{mayne2009robust,kogel2017robust} with uncertain/disturbed linear discrete-time systems under consideration, where the state is estimated using a Luenberger observer.
Combining the observer structure with a tube-based MPC scheme, robust constraint satisfaction and feasibility of the control algorithm were shown.

Along the lines of the second research branch discussed above, namely to update the model using measurement data from the system,
we build on the results of~\cite{BergDenn23}, where a two component controller was used, consisting of funnel MPC in combination with a model-free funnel control feedback loop.
We extend this approach by introducing a general online learning framework in order to continually 
improve the model utilizing the past system data, meaning system output, past model-based predictions,
and applied control signals, both from the model-based and the model-free controller component.
This not only allows to learn and fine-tune parameters of an already detailed model, 
but it is even possible to learn an unknown system without an initially given model. 
Continually improving the model and thereby the prediction capability required in MPC, 
the predictive funnel MPC control signal achieves the control objective for the unknown plant.

Summarizing the above, the main contribution of the present work is the model predictive control \Cref{Algo:LearningRFMPC}.
This algorithm satisfies the control objective introduced in \Cref{Sec:Problem}, i.e., guaranteeing satisfaction of output constraints while safely updating the model using data samples from the plant.
To establish \Cref{Algo:LearningRFMPC},~we
\begin{itemize}
    \item provide sufficient structural conditions to determine a class of feasible models of relative degree one for funnel MPC, cf. \Cref{Def:Model},
    \item introduce structural properties to be satisfied by a learning scheme to update the model for MPC, cf. \Cref{Def:Learning},
    \item prove initial and recursive feasibility of \Cref{Algo:LearningRFMPC} as the main result.
\end{itemize}

The present work is organized as follows.
In \Cref{Sec:Problem}, we provide 
the problem formulation. The control objective, as well as the controller components are introduced.
In \Cref{Sec:SystemModelClass}, we formally introduce the system class and the model class under study.
Anticipating the later considerations, we highlight that the actual system to be controlled, and the model of that system can be
quite different. E.g., it is possible to have a nonlinear system and a linear surrogate model on which the MPC algorithm operates.
\Cref{Sec:Learning} contains the main result of the article.
We introduce the combined three-component controller, establish the corresponding control \Cref{Algo:LearningRFMPC}, and formulate the main result \Cref{Thm:learningFRMPC}, 
which yields initial and recursive feasibility of the proposed control algorithm.
Furthermore, we present a particular learning scheme, and rigorously prove its feasibility.
In the last \Cref{Sec:Simulation}, we illustrate the control \Cref{Algo:LearningRFMPC} with two numerical simulations. \\

\noindent \textbf{Nomenclature}:
$\dB_\eta:=\{x\in\R^n| \Norm{x}\leq \eta\}$ is the closed ball with radius $\eta>0$ around the origin in $\R^n$. 
$\Gl_n(\R)$ is the group of invertible $\R^{n\times n}$ matrices.
For an interval $I\subseteq\R$, $L^\infty(I, \R^{n})$ is the Lebesgue space of
measurable, essentially bounded functions $f\colon I\to\R^n$ with norm $\|f \|_{\infty} =
\esssup_{t \in I} \|f(t)\|$, and
$W^{k,\infty}(I,  \R^{n})$ is the Sobolev space of all functions $f:I\to\R^n$ with $k$-th order weak derivative $f^{(k)}$ and $f,f^{(1)},\ldots,f^{(k)}\in
L^\infty(I, \R^{n})$; and $f|_J$ denotes the restriction of a function $f: I \to \R^n$ to the interval $J \subseteq I$.

\section{Problem formulation} \label{Sec:Problem}
The overall task is output-reference tracking within prescribed bounds on the tracking error. %
To this end, we apply a model predictive controller %
to achieve a superior controller performance while maintaining input and output constraints. However, since model-plant mismatches are unavoidable, we safeguard the MPC component by an additional feedback controller
to guarantee satisfaction of the output-tracking criterion; the funnel controller is a suitable choice for the latter task.
The third component, besides MPC and funnel control, is a learning one, which repeatedly updates the model to reduce the model-plant mismatch and, thus, improves the overall controller performance. 
A major challenge is to ensure proper functioning of the interplay of these three components, which requires a consistency condition for the model updates (see \Cref{Sec:SystemModelClass})~--~the key novelty of our approach in comparison to the recently proposed robust funnel MPC (which combines the first two components) in~\cite{BergDenn23}.
In the present work we establish the general underlying idea of combining these three components. We show how a learning scheme is to be integrated into the control algorithm while we do not specify a particular learning methodology or architecture. 
Nevertheless, we will discuss a variant of parameter identification as one possible learning instance in \Cref{Sec:Learning-scheme}.
\\

\noindent \textbf{The control objective} is output reference tracking within prescribed transient error bounds. 
This means that, for an unknown multi-input multi-output control system of the form
\begin{equation}\label{eq:Sys_Orig}
    \dot{z}(t)   =   F(z(t),d(t))  +  G(z(t),d(t))u(t) , \ \,  y(t)  =  H(z(t)),
\end{equation}
where $z(t) \in \R^l$, bounded disturbance~$d \in L^\infty(\Rp,\R^a)$, and $y(t)\in \R^m$ for $l,a,m \in \N$,
we seek an input~$u(t) \in \R^m$ such that the output~$y$ tracks a given reference signal~$\yref: \Rp \to \R^m$.
We note that the above class covers both single-input single-output and multi-input multi-output systems, where we require the dimension of the input and the output to be the same.
Specific properties of the system parameters $F,G,H$ are introduced in \Cref{Sec:SystemModelClass}.
Moreover, the tracking task is asked to be satisfied with a given precision, i.e., the tracking error $e(t) := y(t) - \yref(t)$ should evolve within 
(possibly time-varying) boundaries given by a so-called funnel function~$\psi$, prescribed by the user.
To be precise, the control objective is that
\begin{equation} \label{eq:ControlObjective}
 \forall \, t \ge 0 \, : \  \| y(t) - y_{\rm ref}(t)\| < \psi(t),   
\end{equation}
i.e.,
the tracking error shall evolve within the %
funnel
\[
    \cF_\psi:= \setdef{ (t,e)\in \Rp\times\R^{m} }{ \Norm{e} < \psi(t) },
\]
which is determined by the function $\psi$
belonging to the set
\[
    \cG =\setdef{ \psi\in W^{1,\infty}(\Rp,\R) }{\text{inf}_{t \ge 0}\, \psi(t)>0 },
\]
see also \Cref{Fig:funnel}.
\begin{figure}[h]
  \begin{center}
  {
    \begin{tikzpicture}[scale=0.35]
    \tikzset{>=latex}
    \filldraw[color=gray!25] plot[smooth] coordinates {(0.15,4.7)(0.7,2.9)(4,0.4)(6,1.5)(9.5,0.4)(10,0.333)(10.01,0.331)(10.041,0.3) (10.041,-0.3)(10.01,-0.331)(10,-0.333)(9.5,-0.4)(6,-1.5)(4,-0.4)(0.7,-2.9)(0.15,-4.7)};
    \draw[thick] plot[smooth] coordinates {(0.15,4.7)(0.7,2.9)(4,0.4)(6,1.5)(9.5,0.4)(10,0.333)(10.01,0.331)(10.041,0.3)};
    \draw[thick] plot[smooth] coordinates {(10.041,-0.3)(10.01,-0.331)(10,-0.333)(9.5,-0.4)(6,-1.5)(4,-0.4)(0.7,-2.9)(0.15,-4.7)};
    \draw[thick,fill=lightgray] (0,0) ellipse (0.4 and 5);
    \draw[thick] (0,0) ellipse (0.1 and 0.333);
    \draw[thick,fill=gray!25] (10.041,0) ellipse (0.1 and 0.333);
    \draw[thick] plot[smooth] coordinates {(0,2)(2,1.1)(4,-0.1)(6,-0.7)(9,0.25)(10,0.15)};
    \draw[thick,->] (-2,0)--(12,0) node[right,above]{\normalsize$t$};
    \draw[thick,dashed](0,0.333)--(10,0.333);
    \draw[thick,dashed](0,-0.333)--(10,-0.333);
    \node [black] at (0,2) {\textbullet};
    \draw[->,thick](4,-3)node[right]{\normalsize$\inf\limits_{t \ge 0} \psi(t)$}--(2.5,-0.4);
    \draw[->,thick](3,3)node[right]{\normalsize$(0,e(0))$}--(0.07,2.07);
    \draw[->,thick](9,3)node[right]{\normalsize$\psi(t)$}--(7,1.4);
    \end{tikzpicture}
  }
  \end{center}
 \vspace{-2mm} 
  \caption{\label{Fig:funnel}Evolution of the error $e$ in a funnel $\mathcal F_{\psi}$ with boundary~$\psi$.
  The figure is based on~\cite[Fig.~1]{BergLe18a}, edited for present purpose. 
  }
\end{figure}
Note that the function $\psi$ is not necessarily monotonically
decreasing and in certain situations, like in the presence of periodic
disturbances, widening the funnel over some time interval might be beneficial.
During safety critical system phases, the funnel diameter will be small, while during non-critical phases the funnel can be widened again to reduce the control effort. 
The funnel function~$\psi\in\cG$ is a design parameter and the specific
application usually dictates constraints on the tracking error and thus
indicates suitable choices.

\ \\
\noindent
\noindent\textbf{Funnel control.}
There exist several methods to achieve reference tracking with $(t,e(t)) \in \mathcal{F}_\psi$ for an unknown system~\eqref{eq:Sys_Orig} satisfying the structural assumptions in \Cref{Def:System-class} (cf. \Cref{Sec:SystemModelClass}). In this work, we will use so-called funnel control, first proposed in~\cite{IlchRyan02b} -- an adaptive high-gain feedback methodology.
For a given reference trajectory~$y_{\rm ref}$ and a funnel function~$\psi$ the instantaneous control input~$u_{\rm fb}$ is calculated via the feedback law
\begin{equation*}
    u_{\rm {fb}}(t,y(t)) = - \psi(t) \frac{y(t) - y_{\rm ref}(t)}{\psi(t)^2 - \|y(t)-y_{\rm ref}(t)\|^2}.
\end{equation*}
This feedback has the following effect:
If the tracking error~$e(t)$ is close to the funnel boundary~$\psi(t)$, then the denominator becomes small and the tracking error is being pushed away from the funnel boundary. 
This effect can be observed in the numerical simulation in \Cref{Fig:Chem_reac_learning_fmpc} in \Cref{Sec:Simulation}.
Further details how this feedback is incorporated in the proposed control method will be provided in \Cref{Sec:Learning}.
In the present work, we use funnel control, but acknowledge that for systems defined in \Cref{Def:System-class} the prescribed performance control methodology proposed in~\cite{bechlioulis2010prescribed,bechlioulis2014low} is equally suitable. 
Both methods are able to achieve the tracking objective.
Note that by $\psi\in\cG$ the tracking error is not forced to converge asymptotically to zero.

\ \\
\noindent \textbf{Robust funnel MPC.}\
To solve the described problem for the unknown system~\eqref{eq:Sys_Orig} relying on predictions of the system's behavior, 
\textit{robust funnel MPC} was proposed in~\cite{BergDenn23}.
Like its underlying funnel MPC scheme~\cite{BergDenn21}, robust funnel MPC is a continuous-time MPC scheme. It consists of two components, one model-based and one model-free, see \Cref{Fig:learning-RFMPC}.
\begin{figure*}[h]
\begin{center}
\resizebox{0.95\textwidth}{!}{
    \begin{tikzpicture}[very thick,%
    scale=0.58,%
    node distance = 9ex,
    box/.style={fill=white,rectangle, draw=black},
    blackdot/.style={inner sep = 0, minimum size=3pt,shape=circle,fill,draw=black},%
    blackdotsmall/.style={inner sep = 0, minimum size=0.1pt,shape=circle,fill,draw=black},%
    plus/.style={fill=white,circle,inner sep = 0,very thick,draw},%
    metabox/.style={inner sep = 3ex,rectangle,draw,dotted,fill=gray!20!white}]
    \begin{scope}[scale=0.5]
        \node (sys) [box,minimum size=9ex,xshift=-1ex, fill=orange!60]  {System};
        \node(FC) [box, below of = sys,yshift=-6ex,minimum size=9ex] {Funnel Controller};
        \node(fork1) [plus, right of = FC, xshift=18ex] {$+$};
        \node(fork2) [plus, left of = FC, xshift=-15ex] {$+$};
        \node(fork3) [blackdot, left of = fork2, xshift=-0ex] {};
        \node(MPC) [box, left of = fork3,xshift=-8ex,minimum size=9ex] {Funnel MPC};
        \node(MPCin1) [minimum size=0pt, inner sep = 0pt, below of = MPC, yshift=4.5ex, xshift=2ex] {};
        \node(MPCin1Desc) [minimum size=0pt, inner sep = 0pt, below of = MPCin1, yshift=5.5ex, xshift=2.5ex] {$y$};
        \node(MPCin2) [minimum size=0pt, inner sep = 0pt, below of = MPC, yshift=4.5ex, xshift=-2ex]{};
        \node(MPCin2Desc) [minimum size=0pt, inner sep = 0pt, below of = MPCin2, yshift=5.5ex, xshift=-2.5ex] {$y_{\rf}$};
        \node(refin) [minimum size=0pt, inner sep = 0pt, below of = MPC, yshift=-2ex, xshift=-2ex] {};
        \node(Mod) [box, above of = MPC,yshift=5ex,minimum size=9ex] {Model};
        \node(fork4) [blackdot, left of = MPC, xshift=-5ex] {};
        \node(fork5) [minimum size=0pt, inner sep = 0pt, below of = fork4, yshift=-3.5ex] {};
        \node(fork9) [blackdot, inner sep = 0pt, right of = fork1, xshift=-2ex ] {};
        \node(fork6) [minimum size=0pt, inner sep = 0pt, below of = fork9, yshift=-2ex] {};
        \node(fork11)[blackdot, inner sep = 0pt, above of = fork9, yshift=6ex ] {};
        \node(fork12)[blackdot, inner sep = 0pt, above of = fork3, yshift=5ex ] {};
        \node(fork13)[blackdot, inner sep = 0pt, above of = fork4, yshift=5ex ] {};
        \node(ML) [box,above of = Mod,minimum size=9ex,xshift=15ex,yshift=8ex] %
            {$\begin{array}{c} \text{Machine}\\ \hspace*{0.55cm}\text{learning}\hspace*{0.55cm} \end{array}$};
        \node(MLu) [minimum size=0pt, inner sep = 0pt, right of = ML, yshift=2ex,xshift=0.6ex] {};
        \node(MLo) [minimum size=0pt, inner sep = 0pt, right of = ML, yshift=-2ex,xshift=0.6ex] {};
        \node(fork8) [minimum size=0pt, inner sep = 0pt, above of = fork1, yshift=-1ex] {};

        \node(fork7) [minimum size=0pt, inner sep = 0pt, below of = MPCin1, yshift=2.5ex] {};
        \draw[-] (MPC) -- (fork3) node[pos=0.2,above, xshift=2ex] {$u_{\rm FMPC}$};
        \draw[->] (refin) -- (MPCin2) node[pos=0.4,left] {};
        \draw[-] (fork3) |- (fork12);
        \draw[->] (fork12) |- (Mod);
        \node(ML_uFMPC)[below of = ML, yshift=5.3ex, xshift=2ex] {};
        \node(ML_yM1)[above of = ML, xshift=2ex] {};
        \coordinate[above of = ML_uFMPC, yshift=3ex] (ML_yM1) ;
        \node(ML_yM2)[above of = ML,yshift=-24, xshift=2ex] {};
        \draw[->] (fork12) -| (ML_uFMPC) node[pos=0.9,right] {$u_{\rm FMPC}$};
        \draw[-] (fork13) |- (ML_yM1.east) ;
        \draw[-] (fork13) -- (fork4) ;
        \draw[->] (ML_yM1.north) -| (ML_yM2) node[pos=0.7,right] {$y_{\rm M}$};
        \draw (Mod) -| (fork13) node[pos=0.2,above] {$y_{\rm M}$}; 
        \draw[->] (fork4) -- (MPC);
        \draw[->] (fork3) -- (fork2.west);
        \draw[-, name path=line4] (sys) -| (fork9)node[pos=0.05,above] {$y$} ;
        \draw[->] (fork9) -- (fork1) node[pos=0.6,right, above] {$+$};
        \draw[-] (fork9) -- (fork6.south);
        \draw[->,name path=line2] (fork5.west) -| (fork1) node[pos=0.9,left] {$-$};
        \draw[->] (fork7.south) -- (MPCin1);
        \path[->,name path=line1] (fork7.west) -- (fork6.east){};
        \path [name intersections={of = line1 and line2}];
        \path [name intersections={of = line1 and line2}];
        \coordinate (S)  at (intersection-1);          
        \path[name path=circle] (S) circle(5.mm);
        \path [name intersections={of = circle and line1}];
        \coordinate (I1)  at (intersection-1);
        \coordinate (I2)  at (intersection-2);
        \tkzDrawArc[color=black, very thick](S,I2)(I1);
        \draw[-] (fork6.east) -- (I2);
        \draw[-] (fork7) |- (I1);
        \draw[-] (fork11) |- (fork9);
        \draw[->] (fork11) |- (MLu) node[pos=0.975,above] {$y$};;
        \draw[->] (fork1) -- (FC) node[midway,above] {$y - y_{\rm M}$};
        \draw[->] (FC) -- (fork2) node[pos=0.4,above, xshift=2ex] {$u_{\rm FC}$};
        \draw[->] (fork2) |- (sys) node[pos=0.73,above] {$u=u_{\rm FMPC} + u_{\rm FC}$};
        \draw (fork4) -- (fork5.south);
        \node(fork10) [blackdot, right of = fork2, xshift=-1ex] {};
        \draw (fork10) |- (fork8.east);
        \draw[->] (ML) -| (Mod) node[pos=0.6,left] {$\begin{array}{l} \text{Model}\\ \text{update}\end{array}$};
        \path[->,name path=line5] (fork8.north) |- (MLo)  node[pos=0.96,below] {$u_{\rm FC}$};
        \path [name intersections={of = line5 and line4}];
        \coordinate (S3)  at (intersection-1);         
        \path[name path=circle3] (S3) circle(5.mm);
        \path[name path=circle3] (S3) circle(5.mm);
        \path [name intersections={of = circle3 and line5}];
        \coordinate (I5)  at (intersection-1);
        \coordinate (I6)  at (intersection-2);
        \tkzDrawArc[color=black, very thick](S3,I5)(I6);
        \draw[-] (fork8.north) -- (I6);
        \draw[->] (I5) |- (MLo);
    \end{scope}
    \begin{pgfonlayer}{background}
        \fill[gray!25] (-18,-8.0) rectangle (9.2,1.8);
        \fill[red!20] (-17.5,-6.4) rectangle (-8,1.5);
        \fill[blue!20] (-5.5,-6.4) rectangle (8,-2.);
        \fill[green!20] (-18.0,2) rectangle (-5,7);
        \node at (-12.5,-8.7) {\color{red!75}{Model-based controller component}};
        \node at (1.5,-8.7) {\color{blue!75}{Model-free controller component}};
        \node at (-1,5.35) {\color{green}{Learning component}};
    \end{pgfonlayer}
\end{tikzpicture}}
  \caption{Design of the proposed three-component controller. The gray box (containing the red (funnel MPC) and the blue (funnel control) structures) represents the two-component controller \emph{robust funnel MPC} proposed in~\cite{BergDenn23}. The green box represents the learning component, which receives the four signals: system output~$y$, model output~$\yM$, funnel MPC control signal~$u_{\rm FMPC}$, and funnel control signal~$u_{\rm FC}$.
  Structural requirements on the learning scheme are introduced in \Cref{Def:Model,Def:Learning}.
  }
  \label{Fig:learning-RFMPC}
  \end{center}
\end{figure*}
The model-based funnel MPC component (red box in \Cref{Fig:learning-RFMPC}) uses  a model of the form
\begin{equation}\label{eq:Mod_Orig}
\begin{aligned}
    \dot{x}(t)  &=  f(x(t))+g(x(t))u(t), \quad   y_{\rm M}(t)  = h(x(t))
\end{aligned}
\end{equation}
where $x(t) \in \R^n$ and $y_{\rm M}(t), u(t) \in \R^m$,
as an approximation of the system~\eqref{eq:Sys_Orig}.
Specific properties of the model parameters $f,g,h$ are introduced in \Cref{Sec:SystemModelClass}.
At time instances $t_k\in \delta\N_0$ with $\delta>0$, the current output~$\yM(t_k)$ of the model~\eqref{eq:Mod_Orig} is measured and 
predictions of the future model behavior are computed over the next time interval of length {$T = N \delta > 0$, $N \in \mathbb{N}$}. Using the {time-varying} \textit{stage cost} 
{$\ell:\Rp\times\R^m\times\R^{m}\to \R_{\geq 0} \cup\{\infty\}$} where $\ell(t,y_{\rm M},u)$ is defined by
\begin{align} \label{eq:stageCostFunnelMPC}
\ell(t,\yM,u) =
    \begin{dcases}
        \frac {\Norm{y_{\rm M}-y_{\rf}(t)}^2}{\psi(t)^2 - \Norm{y_{\rm M}-y_{\rf}(t)}^2} 
        + \l_u \Norm{u}^2,
            & \Norm{y_{\rm M}-y_{\rf}(t)} \neq \psi(t)\\
        \infty,&\text{else},
    \end{dcases}
\end{align}
with regularization parameter~$\lambda_u \ge 0$, a control signal $u_{\rm FMPC}\in L^{\infty}([t_k,t_k+T],\R^m)$
is computed as a solution of a finite horizon optimal control problem.
Additionally, the corresponding predicted output
$y_{\rm M}(\cdot;t_k, x_k, u_{\rm FMPC}) = {h(x(\cdot;t_k, x_k, u_{\rm FMPC}))}\in L^{\infty}([t_k,t_k+T],\R^m)$ is calculated.
Here  $x(\cdot; t_k, {x_k}, u)$ denotes the unique  
solution of~\eqref{eq:Mod_Orig} with initial condition~$x(t_k)=x_k$, which is well-defined for appropriate $f,g,h$ (the model is specified in \Cref{Def:Model})
on the interval $[t_k,t_k+T]$ for appropriate control input~$u$.

The model-free funnel control component (blue box in \Cref{Fig:learning-RFMPC}) computes an 
instantaneous control signal~$u_{\rm FC}$ based on the deviation between the output~$y$ of system~\eqref{eq:Sys_Orig} and the funnel MPC-based predicted~$y_{\rm M}$. 
The sum {$u(t) = u_{\rm FMPC}(t) + u_{\rm FC}(t)$} is then 
applied to the actual system~\eqref{eq:Sys_Orig} {at time instant~$t$}. 
The signal~$u_{\rm FC}$ from the funnel controller is used 
to compensate for occurring disturbances, uncertainties in the model~\eqref{eq:Mod_Orig} and unmodeled dynamics.
It is solely determined by the instantaneous values of the system output~$y$, the funnel function~$\psi$, 
and the prediction~$y_{\rm M}$. Therefore, the model-free component cannot \textit{plan ahead}.
This may result 
in large control values and a rapidly changing control signal if the actual output significantly deviates from its predicted counterpart.
A significant deviation is here to be understood relative to the current size of the funnel given by~$\psi$.
Numerical simulations in~\cite{BergDenn21,berger2019learningbased} show that funnel MPC exhibits a considerably 
better controller performance than pure funnel control.
In order to reduce the {control effort} of the funnel controller, a continuous \emph{activation function} $\beta$ is incorporated and has the following effect: whenever the deviation between predicted~$\yM$
and plant output~$y$ is acceptable (designer's choice), only $\uMPC$ is applied.
A reasonable and simple choice for~$\beta$ is a ReLU-like map, which is zero below a given threshold and linear above.\\

\noindent \textbf{Learning the model} is the third component of the overall task addressed in the present article.
Since funnel MPC exhibits better controller performance, and robust funnel MPC is able to compensate for model-plant mismatches, it is desirable to improve the model {so that, preferably,} the control $u_{\rm FMPC}$ is sufficient to achieve the tracking task {with prescribed performance} for the unknown system while satisfying the input constraints~-- in other words, it is desirable that the funnel controller is inactive.
In \Cref{Def:Model} we identify and establish properties of the learning component such that learning and updating the model preserves the structure necessary for robust funnel MPC~\cite{BergDenn23}.
We emphasize that in the present work we do not focus on a particular learning scheme but establish some structural properties of a learning scheme suitable to be combined with robust funnel MPC.
In \Cref{Sec:Learning-scheme} we discuss a variant of parameter identification as one possible instance of a learning scheme; however, we emphasize that the presented methodology is not restricted to this scheme.
As a result, the robustness w.r.t.
model-plant mismatches of robust funnel MPC even allows to start with ``no model'', e.g., only an integrator chain, and then learn the remaining drift-dynamics.
This 
is considered in a numerical simulation in \Cref{Sec:Simulation}.

\section{System class and model class} \label{Sec:SystemModelClass}
In this section we formally introduce the class of systems to be controlled as well as the class of models to be used in the 
MPC part.
We consider nonlinear multi-input multi-output systems.
Since in our later analysis we mainly refer to the results of~\cite{BergDenn23,BergDenn21}, where the system is considered in so-called Byrnes-Isidori normal form, we assume that the system~\eqref{eq:Sys_Orig} can be transformed (by a transformation which does not need  to be known) into input-output normal form
\begin{subequations} \label{eq:Sys}
    \begin{align}
        \dot y(t) & =  P(y(t),\zeta(t),d(t))  +  \Gamma(y(t),\zeta(t),d(t)) u(t), & y(0)  =  y^0  , \\
        \dot \zeta(t) & =  Q(y(t),\zeta(t),d(t)), & \zeta(0)  =  \zeta^0  ,\label{eq:ZeroDyn}
    \end{align}
\end{subequations}
with control input $u:\Rp\to\R^m$, a bounded internal or external disturbance 
$d : \Rp \to \R^a$, 
and output $y:\Rp\to\R^m$.
The system dynamics are governed by \emph{unknown} nonlinear functions $P \in \cC(\R^m \times \R^{\kappa} \times \R^{a}, \R^m)$ and $\Gamma \in \cC(\R^m \times \R^\kappa \times \R^{a}, \R^{m \times m})$, where the latter is the so-called high-gain matrix. The last equation in~\eqref{eq:Sys} with $Q \in \cC^1(\R^m \times \R^{\kappa} \times \R^a, \R^{q})$ describes the internal dynamics, i.e., the dynamics within the system, which cannot be directly measured at the output, with dimension~$\kappa := l-m \in \N$ ($l$ is the state dimension in~\eqref{eq:Sys_Orig}). 
Referring to system~\eqref{eq:Sys} above, we formally introduce the class of systems under consideration in the present article.
\begin{definition} \label{Def:System-class}
      We say that the system~\eqref{eq:Sys} belongs to the system class~$\cN$, written ${(P,\Gamma,Q,d)} \in \cN$, if $d\in L^\infty(\Rp,\R^a)$,
      the {symmetric part of the} high gain matrix $\Gamma$ is {positive} definite, i.e., $\Gamma(\cdot) + \Gamma(\cdot)^\top > 0$,
      and the internal dynamics satisfy the following bounded-input bounded-state property
      \begin{multline*}
          \forall \, c_0 > 0 \ \exists \, c_1 > 0 \ \forall \, \zeta^0 \in \R^{\kappa} \ \forall \, d \in L^\infty(\Rp,\R^a) 
          \ \forall \, y \in L^\infty(\Rp,\R^m): \\
           \| \zeta^0\| + \|y\|_\infty + \|d\|_\infty \le c_0 \Rightarrow \| \zeta(\cdot;0,\zeta^0,y,d) \|_\infty \le c_1,
      \end{multline*}
where $\zeta(\cdot;0,\zeta^0,y,d):\Rp\to\R^\kappa$ denotes the unique global solution of~\eqref{eq:ZeroDyn}. 
\end{definition}

We remark that the globality of the solution~$\zeta$ is ensured by the imposed conditions.
Note that it is also possible to allow for a negative definite matrix $\Gamma+\Gamma^\top$. Then, we may simply change the sign in the control~\eqref{eq:u} defined below.
The last condition, namely bounded-input bounded-state stability of the internal dynamics, is a common condition for control systems, cf.~\cite{byrnes1988local,khalil2000universal}. 
In our particular case this ensures that the internal states, which cannot be measured directly, are bounded if the output of the system evolves within a bounded funnel around the reference signal and the disturbance is bounded.
For linear systems the function $Q(\cdot)$ consists of three matrices $Q_\zeta \in \R^{\kappa \times \kappa}$, $Q_y \in \R^{\kappa \times m}$ and $Q_d \in \R^{\kappa \times a}$ ($\dot \zeta = Q_\zeta \zeta + Q_y y + Q_d d$),
and the system satisfies the bounded-input bounded-state property, if the matrix $Q_\zeta$ is Hurwitz (all eigenvalues have negative real part).
Such systems are called \emph{minimum phase}, cf.~\cite{liberzon2002output,drucker2023trajectory}. 

For the unknown system~\eqref{eq:Sys} we consider a control-affine surrogate model
\begin{subequations} \label{eq:Mod}
    \begin{align}
\dot y_{\rm M}(t) &= p(\yM(t),\eta(t)) + \gamma(\yM(t),\eta(t)) u(t), && \yM(0) = \yM^0, \label{eq:yM_model} \\
        \dot \eta(t) &= q(\yM(t),\eta(t)), && \ \ \,  \eta(0) = \eta^0, \label{eq:zerodyn_model}      
    \end{align}
\end{subequations}
where properties of the functions
$p,\gamma$ and~$q$ are specified in \Cref{Def:Model}.
Equation~\eqref{eq:zerodyn_model} describes the internal dynamics. 
Referring to~\eqref{eq:Mod}, we introduce the class of feasible models. 
This class is parameterized by the values~$\ModMax$, which ensure that
the dynamics of the model~\ref{eq:Mod} remain uniformly bounded as long as the output $y_{\rm M}$ evolves 
within the funnel given by $\psi$ around the reference trajectory $y_{\rf}$.
This implies, in particular, that the control effort of the model-based controller component 
remains uniformly bounded for all models. 
The tracking task can therefore be successfully performed by the
robust funnel MPC algorithm\cite[Alg.~3.8]{BergDenn23} for a given reference and funnel boundary with the same upper bound~$\umax$ for the maximal control value.

\begin{definition}\label{Def:Model}
    Let $\bar u \ge 0$ and $\bar \gamma > 0$.
    For given $\psi \in \cG$ and $\yref \in W^{1,\infty}(\Rp,\R^m)$
    a model~\eqref{eq:Mod} is said to belong to the model class~$\cM$, 
    written $(p,\gamma,q, \eta^0) \in \cM$, 
    if there exist $\nu \in \N_0$ and $\bar{\eta}\geq 0$ such that the vector $\eta^0\in\R^\nu$ and the functions $p \in \cC^1(\R^m \times \R^{\nu},\R^m)$, 
    $\gamma \in \cC^1(\R^m \times \R^{\nu},\R^{m \times m})$, and $q \in \cC^1(\R^m \times \R^{\nu}, \R^{\nu})$ satisfy the following conditions:
    \begin{enumerate}[label = (M.\arabic{enumi}), ref =  (M.\arabic{enumi}), leftmargin=*]
        \item \label{r1} $\gamma(\rho,\eta)\in\Gl_m(\R)$ and $\|\gamma(\rho,\eta)\| \le \bar\gamma$ for all $ (\rho,\eta) \in \R^m \times \R^{\nu}$, 
        \item \label{BIBO} the solutions of the internal dynamics~\eqref{eq:zerodyn_model} are bounded by~$\bar{\eta}$
        for all trajectories which evolve in the performance funnel defined by~$\psi\in\cG$ around~$y_{\rm ref}$, that is
        \[
          \fa \yM \in  \setdef{ y  \in \con(\Rp,\R^m) }{ \fa t \ge  0: \|y(t) - y_{\rm ref}(t)\| < \psi(t)}:\
            \SNorm{\eta(\cdot;0,\eta^0,\yM)} \leq \bar{\eta},
        \]
        where $\eta(\cdot;0,\eta^0,\yM):\Rp\to\R^\nu$ denotes the global
        solution of~\eqref{eq:zerodyn_model},
        \item \label{Item:BoundModMax} with the set
         \begin{equation}\label{def:SetPsi}
         \Psi=\bigcup_{t\in\Rp} \setdef{\rho\in\R^m }{ \Norm{\rho-y_{\rf}(t)}\leq\psi(t) } 
         \end{equation}
        the functions $p(\cdot)$ and $\gamma(\cdot)^{-1}$ satisfy the uniform bound
        \begin{align}
        \left(\max_{(\rho,\eta)\in \Psi\times \dB_{\bar{\eta}}} \Norm{\gamma(\rho,\eta)^{-1}}\right)
        \left(
        \max_{(\rho,\eta)\in \Psi\times \dB_{\bar{\eta}}} \Norm{p(\rho,\eta)}
        + \| \dot \psi \|_\infty + \| \dot y_{\rm ref}\|_\infty \label{eq:umax_fmpc}
        \right) &\le \bar u .
        \end{align}
    \end{enumerate}
\end{definition}

We comment on the different ingredients of \Cref{Def:Model}.
Note that, although present in~\eqref{eq:Mod}, the initial value $\yM^0$ is not part of the tuple $(p,\gamma,q, \eta^0) \in \cM$
since the model output $\yM(t_k)$ is initialized with the system output~$y(t_k)$ at time instances~$t_k$ in \Cref{Algo:LearningRFMPC}.
Condition~\ref{r1} means that model~\eqref{eq:Mod} has global strict relative degree one with high-gain matrix uniformly bounded by $\bar \gamma > 0$. Restricting the matrix-valued function~$\gamma$ to have maximal norm~$\bar \gamma$ is used in the feasibility proof of \Cref{Algo:LearningRFMPC}, where we use 
\cite[Prop.~6.1, Thm.~3.10]{BergDenn23}. There, a bound on~$\gamma$ is used to have a uniform bound on the derivative of the model's output~$\| \dot y_{\rm M}\|_\infty$ to show uniform boundedness of the funnel control feedback signal.
Note that in order to utilize existing results it is necessary to {include a condition on the norm of~$\gamma$ (like} the parameter~$\bar \gamma$ in~\ref
{r1}) in \Cref{Def:Model}, since~\eqref{eq:umax_fmpc} does not imply a uniform bound on~$\gamma(\cdot)$. Condition~\ref{BIBO} ensures that the internal dynamics~\eqref{eq:zerodyn_model} evolve within an a-priori fixed compact set.
More precisely, if the internal state is bounded at every beginning of the MPC cycle, and the model output evolves within the funnel boundaries around the reference signal, then the internal dynamics evolve within an a-priori-fixed compact set.
This property is used in the feasibility proof of the funnel MPC algorithm, cf.~\cite[Lem.~4.8, Prop.~4.9]{BergDenn21}, and also for the robust funnel MPC algorithm, cf.~\cite[Prop.~5.1, Thm.~3.13]{BergDenn23}.
It was shown in~\cite[Prop.~4.9]{BergDenn21} that there exists a control $u\in L^{\infty}([t_k,t_k+T],\R^m)$
with $\|u\|_\infty \le \umax$ that achieves the tracking task when applied to the surrogate model~\eqref{eq:Mod}, if~\eqref{eq:umax_fmpc} holds.
Therefore, condition~\ref{Item:BoundModMax} ensures that the
prescribed maximal control effort~$\umax$ is sufficient to keep the tracking
error within the funnel boundaries. Since~$\umax$ is fixed in advance, the
models determined by the learning component must be amenable to the achievement
of the tracking objective under control constraints $\|u\|_\infty \le \umax$.

\begin{remark}
    Note that $\bar u = 0 $ is feasible. Such a choice, however, implies 
    $\max_{(\rho,\eta)\in \Psi\times \dB_{\bar{\eta}}} \Norm{p(\rho,\eta)}
        = \| \dot \psi \|_\infty = \| \dot y_{\rm ref}\|_\infty = 0$.
    Moreover, for each~$\bar u > 0$ there exists $\bar{\gamma}>0$ such that the set~$\cM$ is non-empty.
    Namely, setting $p(\cdot) = 0$, $q(\cdot) = 0$ and~$\eta^0 = 0$, we have $(0,\bar \gamma I,0,0) \in \cM$ for $\bar \gamma \ge (\|\dot y_{\rm ref}\|_\infty + \| \dot \psi\|_\infty )/\bar u$.
\end{remark}

\ \\
    There is some difference between the two classes of systems~$\cN$ and models~$\cM$.
     Under certain technical assumptions, cf.~\cite[Cor.~5.7]{ByrnIsid91},  system~\eqref{eq:Sys} as well as model~\eqref{eq:Mod} may be obtained via a transformation from a state-space representation~\eqref{eq:Sys_Orig} or~\eqref{eq:Mod_Orig}, respectively.
    To avoid 
    these technicalities, we restrict ourselves to the input-output representations~\eqref{eq:Sys} and~\eqref{eq:Mod}.
    The system class~$\cN$ encompasses systems with bounded external or internal
    disturbances $d \in L^\infty(\Rp,\R^a)$, while no such disturbances are allowed in the model class~$\cM$.
    The high-gain matrix $\Gamma$ of the system is required to {have a positive definite symmetric part}; in contrast, the high-gain matrix~$\gamma$ in the model is only required to be invertible.

\begin{remark}
    The dimension~$\kappa$ of the internal dynamics of system~\eqref{eq:Sys} is unknown but fixed. 
    In contrast, the dimension $\nu \in \N_0$ of the model's internal dynamics can be considered as a parameter in the learning step.
    This means, in order to improve the model such that it ``explains'' the system measurements, the dimension of the internal state can be varied. Note that $\nu = 0$ (no internal dynamics) is explicitly allowed for the model.
\end{remark}

\begin{remark} \label{Rem:etaBounded}  
        One way to verify satisfaction of~\ref{BIBO} is to apply \cite[Thm.~4.3]{Lanz21}, which states the following.
        Let there exists $V \in \cC^1( \R^\nu , \Rp)$ such that $V(\eta) \to \infty$ as $\|\eta\| \to \infty$, and for $q \in \cC(\R^\nu \times \R^m, \R^\nu)$
        we have
        $
            V'(\eta) \cdot q(\eta,\rho) \le 0 
        $
        for all $\rho \in \R^m$ with $\|\rho\| \le \| \psi\|_\infty + \|\yref\|_\infty$ and all
        $\eta \in \R^\nu$ with $\|\eta\| > \alpha$
        for a predefined~$\alpha > 0$. 
        Then, $\|\eta(t;0,\eta^0,\yM)\|_\infty \le \max\{ \|\eta^0\|, \alpha \}$ for all $t\ge 0$, $\eta^0\in\R^\nu$ and 
        all trajectories $\yM$ which evolve in the performance funnel defined by~$\psi\in\cG$ around~$y_{\rm ref}$, i.e., as in \ref{BIBO}.
        Hence, fixing $V(\cdot)$ and $\alpha>0$ in advance can be used to restrict choices of~$q(\cdot)$ satisfying~\ref{BIBO}.
        We will make use of this fact later in \Cref{Sec:Learning-scheme}, where we discuss a particular learning scheme.
\end{remark}

\section{Learning-based robust funnel MPC} \label{Sec:Learning}
In this section we develop the control algorithm and present our main result. First, we establish the control methodology to achieve the control objective introduced in \Cref{Sec:Problem}.
For the sake of completeness and readability, we recall the robust funnel MPC algorithm proposed in~\cite[Alg.~3.7]{BergDenn23}, adapted to our framework. 
Before, %
we define an abstract learning scheme~$\cL$, which is incorporated in the algorithm.
The idea of the learning component is to use measurement data from the model output~$\yM$, the system output~$y$, the funnel MPC signal~$\uMPC$ and the funnel control signal~$u_{\rm FC}$
to improve the model used for computation of~$\uMPC$ in every MPC sampling interval (cf.~\Cref{Fig:learning-RFMPC}), under the condition that the updated model is still a member of the model class~$\cM$ given in \Cref{Def:Model}.

For example, let 
$(\yM, y, \uMPC, u_{\rm FC})|_{I_0^k} $
be given data collected up to~${t = t_k}$ over the interval $I_0^k:=[0,t_k]$. 
This data is then used to choose $(p,\gamma,q,\eta^0) \in \cM$, i.e.,
\[
\cL\big( (\yM, y, \uMPC, u_{\rm FC})|_{I_0^k} \big) 
=(p,\gamma,q,\eta^0) \in \cM.
\]
The function $\cL$ receives the data, which are available up to the time instant~$t \in \Rp$, and maps it to suitable model functions $(p,\gamma,q,\eta^0) \in \cM$.
Note that $\cL$ does not necessarily use all available data, see the discussion in \Cref{Sec:Learning-scheme}.
Since all four signal types, the outputs~$y$ and~$\yM$ and the control signals~$\uMPC$ and $u_{\rm FC}$, are bounded, the space of data is~$L^\infty$; moreover, all of these four signals have the same dimension~$m \in \N$.
This motivates the following definition.
\begin{definition} \label{Def:Learning}
    For $ \bar u \ge 0$, $\bar \gamma > 0$ we call a function    
$
        \cL : \bigcup_{t\geq 0}L^\infty([0,t],\R^m)^4\to \cM
$
    a feasible learning scheme for robust funnel MPC with respect to the model class~$\cM$, in short notation $\cM$-feasible. 
\end{definition}

\subsection{Control algorithm and main result}
Now, we summarize the reasoning so far in the following algorithm, which achieves the control objective formulated in \Cref{Sec:Problem}. It is a modification of~\cite[Alg.~3.7]{BergDenn23}.
Here, the proper re-initialization {of the model at every time step} done in~\cite{BergDenn23} is substituted by the learning component~$\cL$.
Concerning notation, for an input~$u\in L^\infty([t_k,t_k+T],\R^m)$, we denote by $\yM(\cdot;t_k,y_k^0,\eta_k^0,u)$ the first component of the unique maximal solution of~\eqref{eq:Mod} under the initial condition $(\yM(t_k),\eta(t_k)) = (y_k^0,\eta_k^0) \in \R^{m + \nu}$.
The algorithm involves the following user-defined parameters. 
The time shift~$\delta >0$, the prediction horizon~$T \ge \delta$, 
the parameters $\bar u \ge 0$ and $\bar \gamma > 0$ defining the model class, 
and an initial model $(p_0,\gamma_0,q_0,\eta_0^0) \in \cM$.

\begin{algo}[Learning-based robust funnel MPC]\label{Algo:LearningRFMPC} \ \\
    \textbf{Input}: %
            Instantaneous measurements of the output $y(t)$ of system~\eqref{eq:Sys},
            reference signal $y_{\rf}\in W^{1,\infty}(\Rp,\R^{m})$,
            funnel function $\psi\in\cG$,
            input saturation level $\umax\geq0$ for the MPC component,
            bound~$\bar \gamma>0$ on the norm of~$\gamma(\cdot)$,
            an initial model $(p_0,\gamma_0,q_0,\eta_0^0)\in\cM \neq \emptyset$,
            and an $\cM$-feasible learning scheme~$\cL$.\\
    \textbf{Initialization}: Set %
    time shift $\delta >0$, prediction horizon $T\geq\delta$, and index $k:=0$. 
    Define %
    the time sequence~$(t_k)_{k\in\N_0} $ by $t_k := k\delta$.\\ 
    \textbf{Steps:}
    \begin{enumerate}[label=(\alph{enumi}), ref=(\alph{enumi}), leftmargin=*]
     \item \label{agostep:RobustFMPCFirst} 
     Initialize the model~\eqref{eq:Mod} given by $(p_{k},\gamma_{k},q_{k},\eta_{k}^0) \in \cM$ with the data $(\yM(t_k),\eta(t_k)) = (y(t_{k}),\eta_{k}^0)$.
     \item \textbf{\textsc{Funnel MPC}} \\ 
     For~$\ell$ as in~\eqref{eq:stageCostFunnelMPC} compute a solution $u_{\rm FMPC}$ of the optimal control problem
    \begin{equation}\label{eq:RobustFMPCOCP}
        \mathop
        {\operatorname{minimize}}_{\substack{u\in L^{\infty}(I^k,\R^{m}),\\\SNorm{u}\leq \umax}}  
            \int_{I^k}\ell(t,\yM(t;t_k,{y(t_k)},\eta_k^0,u),u(t))\,\mathrm{d}t,
    \end{equation}
    over the interval~$I^k:=[t_k,t_k+T]$.
    Predict the output~$ y_{\rm M}(t)=\yM(t;t_k,{y(t_k)},\eta_k^0,u_{\rm FMPC})$ of the model~\eqref{eq:Mod} on the interval $[t_k,t_{k+1}]$ 
    and define the adaptive funnel $\phi: [t_k,t_{k+1}]\to\Rpp $ by 
    \begin{equation} \label{alg:eq:vp}
        \phi(t):=\psi(t)- \|  y_{\rm M}(t) - y_{\rf}(t) \|. 
    \end{equation}
    \item \textbf{\textsc{Funnel control}} \\
    Define the funnel control law with $y_{\rm M}|_{[t_k,t_{k+1}]}$ and funnel~$\vp$ as in~\eqref{alg:eq:vp} by
    \begin{equation} \label{alg:eq:FC}
         u_{\rm FC}(t,y(t)) := - \beta\left(\frac{\|y(t)-\yM(t)\|}{\vp(t)}\right) \frac{\vp(t) (y(t)-\yM(t))}{\vp(t)^2 - \|y(t)-\yM(t)\|^2},
    \end{equation}
    for $t\in[t_k,t_{k+1})$, where the activation function
    $\beta\in\con([0,1],[0,\beta^{+}])$ with $\beta^{+}>0$ is such that $\beta(1) = \beta^+$.
    Apply the feedback control $\mu:[ t_k,t_{k+1})\times \R^m\to\R^m$ to system~\eqref{eq:Sys}, given by
        \begin{equation} \label{eq:u}
            \mu(t,y(t)) = u_{\rm FMPC}(t)+ u_{\rm FC}(t,y(t)).
        \end{equation}
            \item \textbf{\textsc{Continual learning}} \\
            Increment~$k$ by $1$,
        find a feasible model for the next sampling interval 
        \begin{equation*}
          \cL\big( (\yM,y,\uMPC,u_{\rm FC})|_{I_0^k} \big) = (p_{k},\gamma_{k},q_{k},\eta_{k}^0)
        \end{equation*}
        with $I_0^k:=[0,t_{k}]$. Then, go to Step~\ref{agostep:RobustFMPCFirst}. 
    \end{enumerate}
\end{algo}
Now we are in the position to formulate the main result of the present article, which extends~\cite[Thm.~3.13]{BergDenn23} by the learning component. 

\begin{theorem} \label{Thm:learningFRMPC}
    Consider a system~\eqref{eq:Sys} with ${(P,\Gamma,Q,d)} \in \cN$ and {initial values $y^0\in\R^m$ and $\zeta^0\in\R^\kappa$}.
    For given $\bar u \geq 0$ and $\bar \gamma>0$ such that $\cM\neq\emptyset$, choose an initial model~\eqref{eq:Mod} with $(p_0,\gamma_0,q_0,\eta_0^0) \in \cM$.
    Let a reference signal $y_{\rf} \in W^{1, \infty}(\Rp,\R^m)$ and a {funnel function} $\psi \in \cG$ be given {such that $\|y^0-\yref(0)\|<\psi(0)$.}
    Further, let~$\cL$ be an $\cM$-feasible learning scheme.
    Then, the learning-based robust funnel MPC \Cref{Algo:LearningRFMPC} 
    with $\delta > 0$ and $T \ge \delta$ is initially and recursively feasible, i.e., 
    at every time instance $t_k := k \delta$ for $k \in \N_0$ the OCP~\eqref{eq:RobustFMPCOCP} has a solution $u_k^* \in L^\infty([t_k,t_k+T],\R^m)$ with $\SNorm{u_k^\star}\leq \umax$, and
     the closed-loop system consisting of system~\eqref{eq:Sys} and the feedback~\eqref{eq:u} has a global solution $(y,\zeta) : [0,\infty) \to \R^m \times \R^\kappa$. 
    In particular, each global solution $(y,\zeta)$ satisfies:
    \begin{enumerate}[label = (\roman{enumi}), ref = (\roman{enumi}), leftmargin=*]
        \item \label{Thm:SignalsBounded} all signals are bounded, in particular, we have $\uMPC, u_{\rm FC}, y \in L^\infty(\Rp,\R^m)$, 
        \item \label{Thm:FunnelProperty} the tracking error $y-y_{\rm ref}$ evolves within prescribed boundaries, i.e.,
        \[
            \forall \, t \ge 0 \, : \ \| y(t)-y_{\rm ref}(t) \| < \psi(t).
        \]
    \end{enumerate}
\end{theorem}
\begin{proof}
To begin with, we show initial feasibility of \Cref{Algo:LearningRFMPC}.
By construction of~$\cM$, and since $\|y^0 - \yref(0)\| < \psi(0)$, invoking~\cite[Prop.~4.9]{BergDenn21} there exists a control $u \in L^\infty([0,T],\R^m)$ with $\|u\|_\infty \le \bar u$ such that $\| \yM(t;0,y^0,\eta_0^0,u) - \yref(t) \| < \psi(t)$ for all $t \in [0,T]$.
Then,~\cite[Thm.~4.5]{BergDenn21} yields that the OCP~\eqref{eq:RobustFMPCOCP} has a solution.
Moreover, $\vp(t) = \psi(t) - \| \yM(t) - \yref(t) \| > 0$ for all $t \in [0,\delta] \subset [0,T]$.
Therefore, the closed-loop system~\eqref{eq:Sys},~\eqref{alg:eq:FC} has a solution on~$[0,\delta]$ according to \cite[Thm.~7]{IlchRyan02b}.
Then, the result~\cite[Thm.~3.13]{BergDenn23} is applicable and yields the existence of a solution of the closed-loop system~\eqref{eq:Sys} and~\eqref{eq:u} on $[0,\delta]$,
where the initialization strategy in~\cite{BergDenn23} is substituted by the initial choice of the model 
in \Cref{Algo:LearningRFMPC}.
Recursive feasibility can then be obtained as follows.
We aim to apply \cite[Prop.~5.1]{BergDenn23}, which 
yields that the OCP in~\eqref{eq:RobustFMPCOCP} has a solution for all~$k \in \N$,
if~$y_{\rm ref}$ and~$\psi$ are defined globally, and the bound~$\bar u \geq 0$ is sufficiently large for the actual model.
The first two conditions are satisfied since~$\yref \in W^{1,\infty}(\Rp,\R^m)$ and $\psi \in \cG$, and the last since the learning scheme is chosen to be $\cM$-feasible.

Next, we have to take care about the possible induced discontinuities in the function~$\vp$ in~\eqref{alg:eq:vp}. These discontinuities, however, occur in both cases either by re-initialization as in~\cite{BergDenn23}, or due to the learning step in \Cref{Algo:LearningRFMPC}.
Therefore, the existence of a global solution of the closed-loop system~\eqref{eq:Sys} and~\eqref{eq:u} follows with the same reasoning as in~\cite[Thm.~3.13]{BergDenn23}.
Namely, $\|y(t_k) - \yref(t_k) \| < \vp(t_k)$ is satisfied for each $k \in \N$ by the previous argument.
Then, \cite[Thm.~7]{IlchRyan02b} yields the existence of a solution of the closed-loop system~\eqref{eq:Sys},~\eqref{alg:eq:FC} on $[t_k, t_k + \delta]$ for all $k \in \N$.
The boundary $\bar{\gamma}$ on $\gamma_k(\cdot)$ and inequality \eqref{eq:umax_fmpc} in property \ref{Item:BoundModMax} ensure that
$\max_{(\rho,\eta)\in \Psi\times \dB_{\bar{\eta}}} \Norm{\gamma_k(\rho,\eta)}$ and $\max_{(\rho,\eta)\in \Psi\times \dB_{\bar{\eta}}} \Norm{p_k(\rho,\eta)}$ with $ \Psi=\bigcup_{t\in\Rp} \setdef{\rho\in\R^m }{ \Norm{\rho-y_{\rf}(t)}\leq\psi(t) }$
are uniformly bounded independent of the model $(p_k,\gamma_k,q_k,\eta_k^0) \in \cM$ generated by the  learning scheme $\cL$.
Thus, 
\[
    \Norm{\dot{y}_{\rm M}}\leq \max_{(\rho,\eta)\in \Psi\times \dB_{\bar{\eta}}} \Norm{p_k(\rho,\eta)}
    +\max_{(\rho,\eta)\in \Psi\times \dB_{\bar{\eta}}} \Norm{\gamma_k(\rho,\eta)}\bar{u}
\]
is uniformly bounded. The bound on the signal generated by the funnel controller computed in~\cite[Thm.~3.13]{BergDenn23}
depends on $\Norm{\dot{y}_{\rm M}}$. 
Repeating the same calculations, one therefore can show that the funnel controller component is uniformly bounded,
i.e., $\SNorm{u_{\rm FC}}<\infty$.
This also yields assertion~\ref{Thm:SignalsBounded}. %
Assertion~\ref{Thm:FunnelProperty} follows by the estimation
   $ \| y(t) - y_{\rm ref}(t)\| = \| y(t) - \yM(t) + (\yM(t) - y_{\rm ref}(t) ) \| 
    < \psi(t) - \|\yM(t) - y_{\rm ref}(t) \| + \|\yM(t) - y_{\rm ref}(t) \| = \psi(t)$
for all~$t \ge 0$.
\end{proof}

\subsection{On learning schemes}\label{Sec:Learning-scheme}

In recent years, data-driven control has gained considerable interest and a plethora of research works have been published.
The respective results can, roughly, be separated into control schemes for linear systems and techniques developed to control nonlinear systems.
Bolstered by successful applications~\cite{Mezi13}, powerful numerical methods such as extended dynamic mode decomposition~\cite{williams:kevrekidis:rowley:2015},
and theoretical advancements --including convergence guarantees in the infinite-data limit~\cite{korda:mezic:2018b} and finite-data error bounds~\cite{KohnPhil24},
even for stochastic control systems~\cite{nuske2023finite}-- the Koopman formalism~\cite{BrunKutz22}, originally introduced in~\cite{koopman1931hamiltonian},
has become a well-established approach for designing data-driven controllers\cite{GoswPale21,OttoRowl21,StraScha24:generator}.
Furthermore, an extension to model predictive control (MPC) has been explored, with recent work providing rigorous closed-loop guarantees\cite{KordMezi18:MPC,BoldGrun25,BoldScha25}.
For linear time-invariant systems, Subspace Predictive Control~\cite{favoreel1999spc} became a popular technique, and the so-called fundamental lemma by Willems and coauthors~\cite{willems2005note} paved the way for direct data-driven control such as DeePC~\cite{coulson2019data}, to name but two.
The structural conditions provided in \Cref{Def:Learning} can be used to define suitable learning algorithms based on the previously discussed techniques~--~for linear as well as for nonlinear systems.

In this section we provide sufficient conditions on the parameters of linear models, which then
guarantee that a respective model is contained in~$\cM$ for fixed~$\bar u,{\bar\gamma}$.
We then invoke these conditions to discuss $\cM$-feasibility of a particular learning scheme.
Here, we consider models
\begin{equation} \label{eq:LearningModel}
    \begin{aligned}
        \dot y_{\rm M}(t) &= R \yM(t) + S \eta(t) + D_1+ \gamma u(t) , & \yM(0)&= \yM^0, \\
        \dot \eta(t) &= Q \eta(t) + P \yM(t) + D_2  , & \eta(0) &= \eta^0.
    \end{aligned}
\end{equation}
In the following, we denote by $\lambda_Q^{+} < 0$ the largest eigenvalue of a symmetric negative definite matrix $Q = Q^\top < 0$.
For a given reference signal $\yref \in W^{1,\infty}(\Rp,\R^m)$ and 
funnel boundary~$\psi \in \cG$ let $\bar \rho := \max_{\rho \in \Psi} \Norm{\rho}$ for the set~$\Psi$ defined in~\eqref{def:SetPsi}. 
Furthermore, for given numbers~$\bar \eta, \bar{d}, \bar{r},\bar{s},\hat \gamma,\bar{p} \ge 0$ 
we define the following set of matrices, where we do not indicate the dependence on the parameters. Let 
\begin{equation*}
    \bar \cK  := 
    \R^{m \times m}  \times \R^{m \times \nu}  \times \Gl_m(\R)  \times \R^{m}  \times \R^{\nu \times \nu}  \times \R^{\nu \times m}  \times \R^{\nu}  \times {\R^\nu} ,
\end{equation*}
and define
\begin{equation}\label{eq:SetOfRetrictions}
    \mathcal{K} := \setdef{(R,S,\gamma,D_1,Q,P,D_2,\eta^0) \in \bar \cK}{\eqref{eq:ConditionsMatrices}} ,
\end{equation}
where
    \begin{align} \label{eq:ConditionsMatrices}
        \| R \| \le \bar{r}, \ 
        \| S \| \le \bar{s}, \ 
       \max\left\{\| \gamma\|, \| \gamma^{-1}\|\right\}  \le \hat \gamma, \ 
       \max \left\{ \|D_1\|,\|D_2\|  \right\}  \le \bar{d}  , \ 
        \lambda_Q^{+}  \le - \frac{\bar{p} \bar \rho + \bar{d}}{\bar \eta}, \ 
        \| P \| \le \bar{p} , \ 
        \| \eta^0 \|  \le \bar \eta.
    \end{align}
For $(R,S,\gamma,D_1,Q,P,D_2,\eta^0) \in \bar \cK$ we define the functions
\begin{equation} \label{eq:p_ParameterAndq_Parameter}
    \begin{aligned}
        p_{R,S,D_1} : \R^m\times\R^{\nu}&\to\R^m, &(y,\eta) &\mapsto Ry + S\eta + D_1, \\
        q_{Q,P,D_2} : \R^m\times\R^{\nu}&\to\R^{\nu},  &(y,\eta) &\mapsto Py + Q \eta + D_2.
    \end{aligned}
\end{equation}
Then, we may derive the following statement. 
\begin{proposition} \label{Lem:IfKThenM}
    Let a reference trajectory $\yref \in W^{1,\infty}(\Rp,\R^m)$ and a funnel function $\psi \in \cG$ be given and set $\bar \rho := \max_{\rho \in \Psi} \Norm{\rho}$, with $\Psi$ defined in~\eqref{def:SetPsi}.
    Further, let an input saturation level $\bar  u\ge 0$ and a uniform bound~$\bar \gamma > 0$ for the high-gain matrix~$\gamma \in \R^{m \times m}$ be given.
    If the parameters $\bar \eta, \bar{d}, \bar{r},\bar{s},\hat \gamma,\bar{p} \ge 0$ satisfy
    \begin{equation}\label{eq:DefModMax}
        \hat \gamma \left( \bar r \bar \rho + \bar s \bar \eta + \bar d + \|\dot y_{\rm ref}\|_\infty
         + \| \dot \psi \|_\infty \right) \le \bar u, \qquad {\hat \gamma \le \bar \gamma}
    \end{equation}
    then the set~$\mathcal{K}$ defined in~\eqref{eq:SetOfRetrictions} satisfies the implication
    \begin{equation*}
        (R,S,\gamma,D_1,Q,P,D_2,\eta^0) \in \mathcal{K}
        \implies 
        (p_{R,S,D_1}, \gamma, q_{Q,P,D_2}, \eta^0) \in \cM, 
        \end{equation*}
    where $p_{R,S,D_1}$ and $q_{Q,P,D_2}$ are given in~\eqref{eq:p_ParameterAndq_Parameter}.
\end{proposition}
\begin{proof}
The proof consists of three parts, with one part devoted to each condition \ref{r1}~--~\ref{Item:BoundModMax}
in Definition~\ref{Def:Model}, respectively. \\
\emph{Step one~\ref{r1}.} The choice $\gamma \in \Gl_m(\R)$ together with~\eqref{eq:ConditionsMatrices} and $\hat \gamma \le \bar \gamma$ in~\eqref{eq:DefModMax} satisfies condition~\ref{r1}.
\\
\emph{Step two~\ref{BIBO}.} We show for the given parameter ${\bar \eta \ge 0}$, that for any tuple of parameters contained in~$\mathcal{K}$, the solution of~\eqref{eq:LearningModel} satisfies $\|\eta(t;0,\eta^0,\yM) \| \le \bar \eta$ for all $t \ge 0$ and 
$  \yM \in  \setdef{ y  \in \con(\Rp,\R^m) }{ \fa t \ge  0: \|y(t) - y_{\rm ref}(t)\| < \psi(t)}$
. Note that any such $\yM$ satisfies $\| \yM\|_\infty \le \bar \rho$.
In virtue of \Cref{Rem:etaBounded} we calculate for $t \ge 0$ that
\begin{equation*}
    \begin{aligned}
        \dd{t} \tfrac{1}{2} \| \eta(t;0,\eta^0,\yM) \|^2  
        &= \eta(t;0,\eta^0,\yM)^\top \Big( Q \eta(t;0,\eta^0,\yM) + P \yM(t) + D_2 \Big) \\
        & \le \| \eta(t;0,\eta^0,\yM) \| \Big( \lambda_Q^+ \| \eta(t;0,\eta^0,\yM) \| + \bar{p} \bar \rho + \bar{d} \Big),
    \end{aligned}
\end{equation*}
which by $\lambda_Q^+<0$ is non-positive for $\| \eta(t;\eta^0,\yM) \| \ge (\bar{p} \bar \rho + \bar{d}) / |\lambda_Q^+|$. 
Therefore, \cite[Thm.~4.3]{Lanz21} yields
\[
    \fa t\ge 0:\ \|\eta(t; 0,\eta^0,\yM)\| \le \max\{ (\bar{p} \bar \rho + \bar{d}) / |\lambda_Q^+|, \|\eta^0\| \}.
\]
By assumption~\eqref{eq:ConditionsMatrices} we have $\| \eta^0\| \le \bar \eta$ and $|\lambda_Q^+| \ge (\bar{p} \bar \rho + \bar{d}) / \bar \eta$, hence it follows $\|\eta(t;0,\eta^0,\yM) \| \le \bar \eta$ for all $t \ge 0$. \\
\emph{Step three~\ref{Item:BoundModMax}.} 
In virtue of the inequalities~\eqref{eq:ConditionsMatrices} and~\eqref{eq:DefModMax}, we may estimate
 \begin{align*}
        \max_{(\rho,\eta)\in \Psi\times \dB_{\bar{\eta}}} \Norm{\gamma(\rho,\eta)^{-1}}
        \left(
        \max_{(\rho,\eta)\in \Psi\times \dB_{\bar{\eta}}} \Norm{p(\rho,\eta)}
        + \| \dot \psi \|_\infty + \| \dot y_{\rm ref}\|_\infty 
        \right)  \le \hat \gamma(\bar r \bar p + \bar s \bar \eta + \bar d + \|\dot y_{\rm ref}\|_\infty + \| \dot \psi \|_\infty)  &\le \bar u.
        \end{align*}
This completes the proof.
\end{proof}
With the set of parameters~$\cK$, the functions $p_{R,S,D_1},q_{Q,P,D_2}$ defined in~\eqref{eq:p_ParameterAndq_Parameter}, and \Cref{Lem:IfKThenM}, we may define a learning scheme $\cL$ mapping from $\bigcup_{t\geq 0}L^\infty([0,t],\R^m)^4$
to the subset
\[
\setdef{ (p_{R,S,D_1},\gamma,q_{Q,P,D_2},\eta^0) }
{(R,S,\gamma,D_1,Q,P,D_2,\eta^0)  \in  \cK }  
\]
of $\cM$, defined by
\begin{align*}
    \cL:\ & ((\yM,y, u_{\rm FMPC}, u_{\rm FC})|_{[0,t]}) 
    \mapsto (p_{R,S,D_1},\gamma,q_{Q,P,D_2},\eta^0)
\end{align*}
for some $t\ge 0$, where $(p_{R,S,D_1},\gamma,q_{Q,P,D_2},\eta^0)$ is determined by the solution of an optimization problem involving discrete measurements 
of the system data~$y$ and the  applied control signals $u_{\rm FMPC}$ and $u_{\rm FC}$ 
at time instances~$i\tau$ with $\tau>0$, $i\in\N$ and $i \tau\leq t$ of the form
\begin{equation}\label{eq:LearningOP}
    \begin{alignedat}{2}
             \hspace{1cm}&\hspace{-1cm}\mathop {\operatorname{minimize}}_{\substack{(R,S,\gamma,D_1,Q,P,D_2,\eta^0) \in \cK}}  \
             J((y,z)|_{[0,t]})\\
            \text{s.t.}\ 
                      &{\makebox[\widthof{$z(i\tau)$}]{$z(0)$}} =  z^0 \text{and for all }i\leq t/\tau:\\
                      &z(i\tau) =  \chi(\tau; z((i-1)\tau ),(u_{\rm FMPC}+u_{\rm FC})((i-1)\tau)),
    \end{alignedat}
\end{equation}
where $J(\cdot)$ is a suitable cost function, $z=(\tilde{y}_{\rm M},\eta)$ denotes the states of the linear model~\eqref{eq:LearningModel}, 
and the expression ${\chi(\cdot; z((i-1)\tau ),(u_{\rm FMPC}+u_{\rm FC})((i-1)\tau))}$ denotes its solution under the initial condition  $\chi(0) = z((i-1)\tau )$ 
and with constant control~$u(\cdot)\equiv(u_{\rm FMPC}+u_{\rm FC})((i-1)\tau)$.
In the following we discuss some possible choices for the cost function~$J(\cdot)$. 

\begin{enumerate}[label = (\roman{enumi}), ref = (\roman{enumi})]
    \item\label{Item:LearnEntireHistory} $J((y,z)|_{[0,t]}) := \sum_{i=0}^{\lfloor t/\tau\rfloor} \xi_i \| \tilde{y}_{\rm M}(i\tau) - y(i\tau)\|^2$ with weights $\xi_i \ge 0$. 
    The idea is to find a model in the set~$\cK$ which minimizes the weighted squared measured output errors.
    The weights $\xi_i$ reflect the relative importance of the measurements $y(i\tau)$. 
    In certain cases it might be beneficial to weight data points that are far in the past lower than current data points. 
    By choosing $\xi_i>0$ for all $i>0$ all measured past data is taken into account.
    With increasing runtime of the algorithm, this results in a growing complexity of the optimization problem,
    computation time,  and required memory space for the measurements.
    Therefore, this is not suitable in practice.
    Thus, it is beneficial to use a moving horizon estimation approach and only take the last $N$ measurements into account and set $\xi_{i}=0$ for $i<\lfloor t/\tau\rfloor-N$.
    In application, one has to find a good balance between considering many data points (large $N$), thus having a probably more accurate model,
    and low computation time and memory requirements (small $N$).
    
    \item If the computation of the solution of the optimization problem has to be done very quickly,
    it is also possible to  only consider the last measurement~$y(\lfloor t/\tau\rfloor\tau)$.
    Thus, one might choose the cost function   $J((y,z)|_{[0,t]}) := \| \tilde{y}_{\rm M}(\lfloor t/\tau\rfloor\tau) - y(\lfloor t/\tau\rfloor\tau)\|^2$.
    The idea is to find a model, which best explains the last MPC period in terms of output error,
    i.e., a model on the prediction interval $[t_k,t_{k+1}]$, so that with $\tau = \delta$
     the error $\| \tilde{y}_{\rm M}(t_{k+1}) - y(t_{k+1})\|$ at the end of the interval is minimal. 
    
    \item In addition, it is worth considering to include regularization terms for the model parameters in the cost function. 
    For the parameters $\cK_i=(R_i,S_i,\gamma_i,D_{1,i},Q_i,P_i,D_{2,i},\eta^0_i) \in  \cK$ 
    one could either penalize the weighted distance of $\cK_i$ to an a priori known tuple of parameters 
    $\cK^\star=(R^\star,S^\star,\gamma^\star,D^\star,Q^\star,P^\star,D_2^\star,{\eta^0}^\star)$ and thus allow only small adaptions of the a priori known model or 
    penalize the change of parameters $\cK_i$ such that the model does only change slightly between two learning steps. 
    This results in a cost function of the form 
    $J((y,z)|_{[0,t]}) := \sum_{i=0}^{\lfloor t/\tau\rfloor} \big (\xi_i \|  \tilde{y}_{\rm M}(i\tau) - y(i\tau)\|^2 + \sum_{j=1}^8k_i^j\|(\cK_i^j-\tilde{\cK}^j)\|\big )$,
    where $\tilde{\cK}=\cK^\star$ or $\tilde{\cK}=\cK_{i-1}$ and
    with weights $\xi_i ,k_i^j \ge 0 $.
    Here the expressions $\cK_i^j, \tilde \cK_i^j$ with $j=1,\ldots,8$ refer to the  $j^{\rm th}$ entry of the tuple $\cK_i$, $\tilde{\cK}_i$, respectively; for instance, $\cK_i^2 = S_i$.
\end{enumerate}

\section{Numerical simulation} \label{Sec:Simulation}
In this section we provide two numerical simulations to illustrate \Cref{Algo:LearningRFMPC}.
In the first example, we simulate the tracking task for a system of relative degree one, which is contained in the system class~$\cN$, and the corresponding surrogate model belongs to~$\cM$, for a given~$\bar u$.
The second example goes beyond the system class~$\cN$, it is a system of relative degree two.
We show that, although no theoretical results on its functioning are available yet, \Cref{Algo:LearningRFMPC} is successful in this case. 

\subsection{Relative degree one}
To illustrate the application of \Cref{Algo:LearningRFMPC}, we consider a model of an exothermic chemical reaction 
which was also used in~\cite{BergDenn21} to study funnel MPC and in~\cite{BergDenn23} to study robust funnel MPC.
The dynamics for one reactant~$\zeta_{1}$, a product~$\zeta_{2}$, and reactor temperature~$y$ is described by 
a system~\eqref{eq:Sys}
where $m=1$, $\kappa=2$, $a=0$, $\Gamma = 1$, $P(y,\zeta) = b_1 \alpha(y,\zeta) - b_2 y$ and 
$Q(y,\zeta) = \begin{smallbmatrix}
    c_1&0\\
    0&c_2
\end{smallbmatrix} 
\alpha(y,\zeta) + c_3(\zeta^{\rm in}-\zeta)$,
with the Arrhenius law $\alpha:\Rpp\times\Rp^2\to\Rp$ given by
$\alpha(y,\zeta):=k_0\me^{-k_1/y}\zeta_1$,
and parameters $k_0$, $k_1$, $b_1$, $b_2>0$, $c_1<0$, $c_2\in\R$, $c_3>0$, $\zeta^{\rm in}\in\Rp^2$.
The control objective is to steer the reactor's temperature~$y$ to a desired constant value $y_{\rf}\equiv\bar{y}$
within prescribed funnel boundaries $\psi\in\cG$,~i.e., $\Norm{y(t)-\bar y} < \psi(t)$ for all $t\geq0$.
For the funnel MPC component of~\Cref{Algo:LearningRFMPC} we consider linear models of the form~\eqref{eq:LearningModel}
with $R,  D_1\in\R$, $S,D_2^\top, P^\top\in\R^{1\times 2}$, and $Q\in\R^{2\times 2}$.
To learn the model from the measured data we use linear regression subject to the constraints introduced in \Cref{Def:Model,Def:Learning}. Hence, feasibility of the data-based models is guaranteed by \Cref{Lem:IfKThenM}.
We assume $\gamma= 1$ and as initial model we choose 
$R=  D_1=0\in\R$,
$S=D_2^\top= P^\top=0\in\R^{1\times 2}$,
$Q=0\in\R^{2\times 2}$, and $\eta^0=(0.02,0.9)$,
which represents an integrator chain with decoupled internal dynamics.
To improve this (deliberately poorly chosen) model over time, we adapt the matrices over a compact set~$\cK$
as in~\eqref{eq:SetOfRetrictions} at every fifth time step $t_k$
by minimizing the plant-model mismatch based on the data of the last system output $y(t_{k-1})$, i.e.,
we solve the optimization problem
\begin{equation*}
    \begin{aligned}
        \mathop{\operatorname{minimize}}_{\substack{(R,S,1,D_1,Q,P,D_2,\zeta(0))\in \cK}}  
            \|{y_{\rm M}(t_k)-y(t_k)}\|^2 \quad 
        \textrm{s.t.} \quad 
       \frac{\textrm{d}}{\textrm{d} t} \begin{pmatrix}
        y_{\rm M}(t)\\
        \eta(t)
    \end{pmatrix}  &=  
       \begin{bmatrix}
        R & S\\
        P & Q
        \end{bmatrix}
       \begin{pmatrix}
             y_{\rm M}(t)\\
            \eta(t)
        \end{pmatrix} 
         + 
       \begin{bmatrix}
        1\\
        0
        \end{bmatrix}
         u(t) +
       \begin{bmatrix}
        D_1 \\
        D_2 
        \end{bmatrix}, \\
        \begin{pmatrix}
                y_{\rm M}(t_{k-1})\\
                \eta(t_{k-1})
            \end{pmatrix}  &= 
         \begin{pmatrix}
             y(t_{k-1})\\
             \zeta(0)
         \end{pmatrix},
    \end{aligned}
\end{equation*}
where $u(t)= u_{\rm FMPC}(t_{k-1})+u_{\rm FC}(t_{k-1})$ which was applied to the model at the last time step~$t_{k-1}$
and $\zeta(0) = (\zeta_1(0),\zeta_2(0))$ is the vector of initial concentrations of the substances~$\zeta_1$ and~$\zeta_2$.

For the simulation we choose the funnel function  $\psi(t)=100\me^{-2t}+1.5$.
As in~\cite{BergDenn23,BergDenn21}, the initial data is $(y^0,\zeta_1(0),\zeta_2(0)):=(270,0.02,0.9)$, the reference signal is $y_{\rf}\equiv 337.1$, and the 
parameters are
     $c_1=-1$,
     $c_2=1$,
     $c_3= 1.1$,
     $k_0=\me^{25}$,
     $k_1=8700$,
     $b_1=209.2$,
     $b_2=1.25$,
     $\zeta_1^{\rm in}=1$,
     $\zeta_{2}^{\rm in}=0$.
In this example we choose~$\lambda_{u}=10^{-4}$, prediction horizon $T=1$, and time shift $\delta=0.1$ for  the funnel MPC component.
We choose for the set~$\cK$ as in~\eqref{eq:SetOfRetrictions}
the parameters as 
$\bar{r}=1.3$,
$\bar{s}=1.4$,
$\bar{\eta} =0.91$,
$\bar{\rho}=408.6$,
$\bar{\gamma}=1$,
$\bar{p}=1/400$,
$\bar{d} = 2.5$,
and  $\hat{\gamma}=1$.
In accordance with \Cref{Lem:IfKThenM}, 
we restrict the funnel MPC control to $\SNorm{u_{\rm FMPC}}\leq \bar{u}:=735$
so that~\eqref{eq:ConditionsMatrices} is satisfied for $\bar{\gamma}=1$. 
Thus, the learning scheme is ${\cM}$--feasible. Due to discretization, only step functions with a constant step length of $0.1$ were considered to solve the OCP~\eqref{eq:RobustFMPCOCP}.
The activation function of the funnel controller is constant $\beta\equiv 10$.
\begin{figure}
\centering
    \begin{minipage}[t]{0.48\textwidth}
    \centering
   \includegraphics[width=\textwidth]{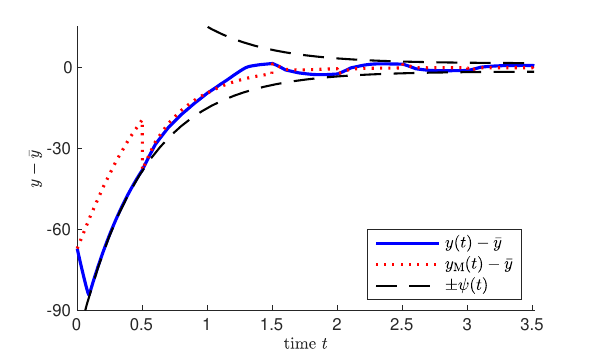}
\end{minipage}
\begin{minipage}[t]{0.48\textwidth}
\centering
    \includegraphics[width=\textwidth]{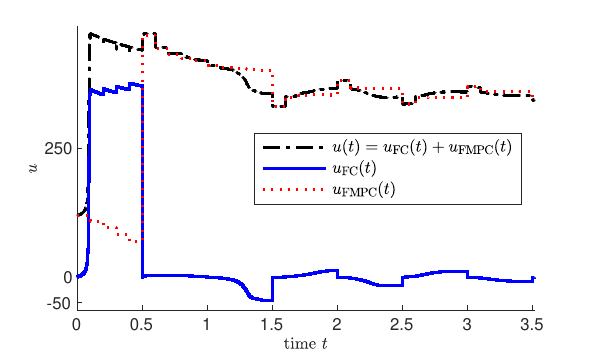}
\end{minipage}
    \caption{Application of learning based robust funnel MPC to the exothermic chemical reactor system.}
    \label{Fig:Chem_reac_learning_fmpc}
\end{figure}

\Cref{Fig:Chem_reac_learning_fmpc} shows the control signals and the system and model output errors, respectively.
It is evident that both $y_{\rm M}-y_{\rf}$
and $y-y_{\rf}$ remain within the predefined funnel boundaries~$\psi$.
Before the first learning step ($t \in [0, 0.5)$) the tracking error $y-y_{\rf}$ and the
predicted error~$y_{\rm M}-y_{\rf}$  diverge due to the poor quality of the initial model.
However, since the tracking error is not close to the funnel boundary,
the funnel controller remains inactive in the beginning and 
only reacts when the tracking error is close to the boundary.
After the first learning step, the general direction of the predicted tracking error is consistent 
with the actual tracking error. The funnel controller still has to compensate for the model inaccuracies
in order to guarantee that the tracking error remains within the boundaries, but with a significantly smaller contribution to the control signal.
After each learning step, the model output jumps to the system output 
due to the newly updated model. 
The control signal $u_{\rm FC}$ is zero after each learning step 
since the system and model output  coincide, 
and it becomes larger afterwards to compensate for the model inaccuracy.
After $t=1.5$ the system output is close to the desired constant reference signal.
Thenceforth, the linear model is adequate to predict the system behavior and the control signal computed by funnel MPC 
is sufficient to achieve the tracking objective.
The funnel controller only has to slightly compensate for model errors. 

In a second simulation, we add an artificial
disturbance $d$ to the differential equation governing the reactor temperature,
i.e., we replace the function $P$ by $\tilde{P}(t,y,\zeta) = b_1 \alpha(y,\zeta) - b_2 y +d(t)$. The disturbance is unknown to the controller and for the simulation we choose the periodic disturbance $d(t)= 25\cdot \sin(15\cdot t)$ and 
leave the controller as it is. The results are depicted in \Cref{Fig:Chem_reac_learning_fmpc_dist}.
\begin{figure}
\centering
    \begin{minipage}[t]{0.48\textwidth}
    \centering
   \includegraphics[width=\textwidth]{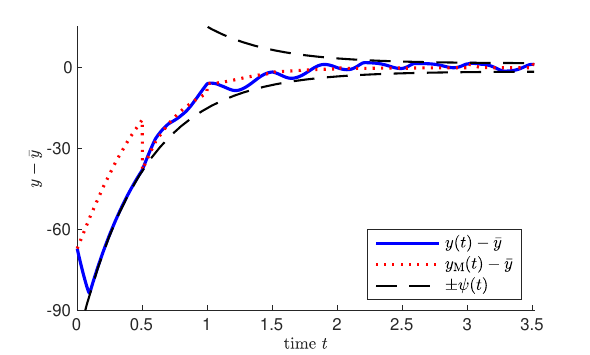}
\end{minipage}
\begin{minipage}[t]{0.48\textwidth}
\centering
    \includegraphics[width=\textwidth]{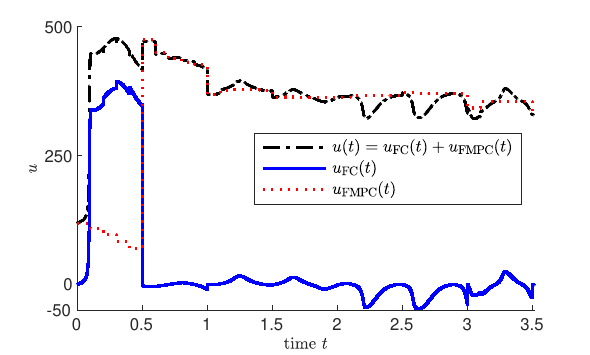}
\end{minipage}
    \caption{Application of learning based robust funnel MPC to the exothermic chemical reactor system in the presence of periodic disturbances.}
    \label{Fig:Chem_reac_learning_fmpc_dist}
\end{figure}
As one can see, the combined controller is able to achieve the control objective 
and the overall system behavior is similar to the one of the undisturbed system.
Before the initial learning step, the funnel controller has to compensate for the inaccuracies of the model. 
Already after the first update of the model, the principal portion of the control signal is generated by the MPC component.
In contrast to the undisturbed case, the funnel controller remains active also when the system output is close to the desired constant reference.
It has to compensate for the unknown periodic disturbances.

We note that this example merely serves to illustrate 
that robust funnel MPC can be combined with $\cM$--feasible learning techniques.
We do not claim that the learning algorithm used is superior to other methods. 

\subsection{Second numerical example: higher relative degree.} \label{Sec:Simulation_RelDeg2}
In this section, we present some promising preliminary results on extending 
the learning-based robust funnel MPC~\Cref{Algo:LearningRFMPC} to a larger system class.
Numerical simulations in~\cite{BergDenn21} suggest that the funnel MPC algorithm is also applicable to systems of higher relative degree, 
meaning systems of the form
\begin{align*}
         y^{(r)}(t) &= P(y(t),\dot{y}(t),\ldots,y^{(r-1)}(t),\zeta(t))
         + \Gamma(y(t),\dot{y}(t),\ldots,y^{(r-1)}(t),\zeta(t)) u(t), \\
        \dot \zeta(t) &= Q(y(t),\dot{y}(t),\ldots,y^{(r-1)}(t),\zeta(t)),
\end{align*}
with $r>1$.
Restricting the class of admissible funnel functions, utilizing an adapted cost function,
and incorporating so-called feasibility constraints in the optimization problem,
feasibility of funnel MPC for this system class was proved in~\cite{BergDenn22}.
It is an open problem whether it is sufficient to utilize the far simpler optimization~\eqref{eq:RobustFMPCOCP} with cost function~\eqref{eq:stageCostFunnelMPC} 
and without the mentioned constraints.
Accordingly, only the case of a relative degree one system was considered for robust funnel MPC in~\cite{BergDenn23}. 
A generalization to systems with higher relative degree has yet to be found.

To illustrate that, nevertheless, the learning-based robust funnel MPC \Cref{Algo:LearningRFMPC} 
shows promising results for this lager system class with fixed relative degree $r>1$,
we consider the example of a mass-on-car system from~\cite{SeifBlaj13} which was 
was also used in~\cite{BergDenn21}.
The mass $m_2$ is mounted on a car with mass $m_1$ via a spring and damper system with spring constant $k>0$ and damper constant $d>0$,
and moves on a  ramp which is inclined by the angle~$\vartheta\in [0,\tfrac{\pi}{2})$.
The  car can be controlled via the force $u$ acting on it.
The situation is depicted in~\Cref{Fig:Mass_on_car}.
\begin{figure}[h!]
    \begin{center}
    \includegraphics[trim=2cm 4cm 5cm 15cm,clip=true,width=5.5cm]{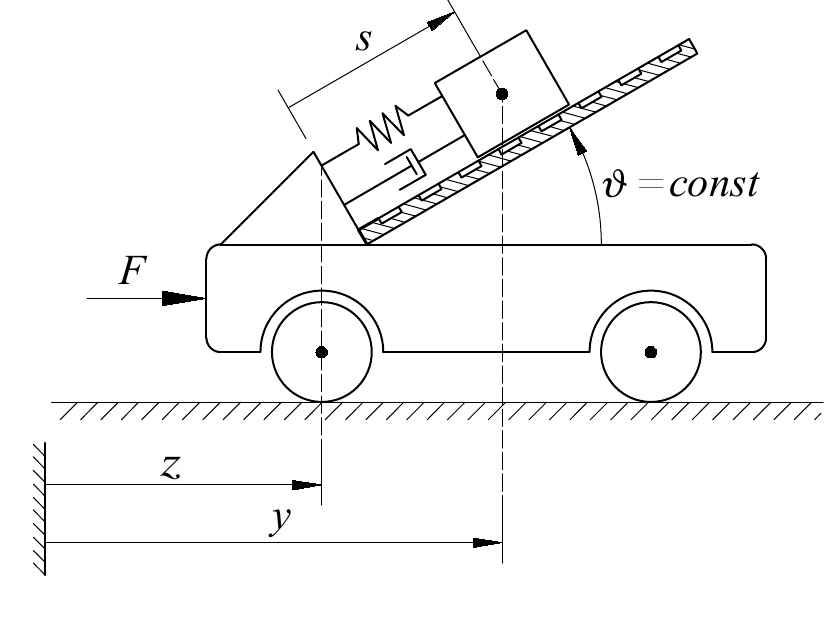}
    \end{center}
    \vspace*{-3mm}
    \caption{Mass-on-car system. The figure is based on the respective figures in~\cite{BergIlch21}, and~\cite{SeifBlaj13}.}
    \label{Fig:Mass_on_car}
\end{figure}

The system can be described by the equations 
\begin{align}\label{eq:ExampleMassOnCarSystem}
    \begin{bmatrix}
        m_1 + m_2& m_2\cos(\vartheta)\\
        m_2 \cos(\vartheta) & m_2
    \end{bmatrix}
    \begin{pmatrix}
        \ddot{z}(t)\\
        \ddot{s}(t)
    \end{pmatrix}
     + 
    \begin{pmatrix}
        0\\
        k s(t) +d\dot{s}(t)
    \end{pmatrix}
     = 
    \begin{pmatrix}
        u(t)\\
        0
    \end{pmatrix}.
\end{align}
The horizontal position of the car is $z(t)$, and  the relative position of the mass on
the ramp at time $t$ is $s(t)$. The output~$y$ of the system is 
\[
    y(t) = z(t) +s(t)\cos(\vartheta),
\]
the horizontal position of the mass on the ramp.
For the system we chose the same parameters $m_1 = 4$, $m_2 = 1$, $k = 2$, $d = 1$, $\vartheta=\pi/4$, and initial values
$z(0) = s(0) = \dot{z}(0) = \dot{s}(0) = 0$ as in \cite{BergDenn21}. Then, as shown there, the system has a relative degree of $r=2$.
The objective is tracking of the reference signal $y_{\rf}(t) = \cos(t)$ such that the tracking error
$y(t)-y_{\rf}(t)$ evolves within the prescribed performance funnel given by the function $\psi\in\cG$ with $\psi(t) = 5\me^{-2t}+0.2$.

For the model-based component of the controller, we solve the optimal control problem~\eqref{eq:RobustFMPCOCP} with stage cost~\eqref{eq:stageCostFunnelMPC}.
The prediction horizon and time shift are selected as $T=0.6$ and $\delta=0.06$, resp.
We restrict the funnel MPC control to $\SNorm{u_{\rm FMPC}}\leq 25$ and choose the parameter $\lambda_u =10^{-2}$ for the stage cost $\ell$.
Due to discretization, only step functions with a constant step length of $0.06$ were considered to solve the OCP~\eqref{eq:RobustFMPCOCP}.
For the model-free component of the controller, we use the funnel controller for systems with relative degree two from~\cite{BergIlch21}
instead of~\eqref{alg:eq:FC}. This component takes the form
\begin{equation*} 
    \begin{aligned}
        w(t)&=\frac{\dot{e}_{\rm M}(t)}{\phi(t)}+\alpha\rbl\tfrac{e_{\rm M}(t)^2}{\phi(t)^2}\rbr\frac{e_{\rm M}(t)}{\phi(t)},\\
        u_{\rm FC}(t)&=-\alpha(w(t)^2)w(t),
    \end{aligned}
\end{equation*}
with $e_{\rm M}(t)=y_{\rm M}(t)-y(t)$ and  $\alpha(s)=\frac{1}{1-s}$ for $s\in[0,1)$.
The controller is applied to the system with a step size $0.6\cdot 10^{-5}$.
Similar to~\cite{berger2019learningbased}, where this problem was studied in the context of model identification during runtime,
for the learning component we assume some knowledge about the structure of the system, but only limited information about the parameters 
and the initial value.
We assume to know $\vartheta\in(0,2\pi]$,  $m_1\in [1,6]$, $m_2\in[0.5,1.5]$, $k\in[1,3]$, $d\in[0.5,1.5]$ and
$z^0:=(x(0),\dot{x}(0),s(0),\dot{s}(0))\in[-2.5,3.5]\times [-2,2]\times [-2.75,3.25]\times [-2,2]$. As initial model all model parameters where chosen equal to $1$ and the initial
state $z^0=(0,1,0,1)$.
To learn the system parameters, we take measurements of the input-output data $((u_{\rm FMPC}+u_{\rm FC})(i\tau),y(i\tau))$
for $\tau =0.006$ and $i\in\N_{0}$ and at every time $jT=100j\tau$ for $j\in\N$ we solve the optimization problem
\begin{align*}
    \mathop
            {\operatorname{minimize}}_{\vartheta,m_1,m_2,k,d,z^0
            }  \quad
            & \sum_{i=0}^{100j}\Norm{\tilde{y}_{\rm M}(i \tau) - y(i \tau)}^2\\
            \text{s.t.} \ \ 
                       z(0) &=  z^0 \text{ and for all }i=1,\ldots, 100j:\\
                       z(i\tau )  &=  \chi(\tau; z((i-1)\tau ),(u_{\rm FMPC}+u_{\rm FC})((i-1)\tau)),\\
                      \tilde{y}_{\rm M}(i\tau )  &=  [1,\cos(\vartheta),0, 0] \, z(i\tau),
\end{align*}
where $z=(x,\dot{x},s,\dot{s})$ denotes the state of the mass on car system~\eqref{eq:ExampleMassOnCarSystem} and
$\chi(\cdot; z((i-1)\tau ),(u_{\rm FMPC}+u_{\rm FC})((i-1)\tau))$ denotes its solution under the initial condition $\chi(0) = z((i-1)\tau )$ and with constant control~$u(\cdot)\equiv(u_{\rm FMPC}+u_{\rm FC})((i-1)\tau)$.
Since only interval $[0,4]$ is considered for the simulation, the entire history of input-output data is considered in the optimization problem instead of a moving horizon approach
as discussed in~\Cref{Sec:Learning-scheme}~\ref{Item:LearnEntireHistory}.

All simulations are performed on the time interval $[0,4]$ with \textsc{Matlab} and the toolkit \textsc{Casadi}
and are depicted in~\Cref{Fig:MassOnCar_learning_fmpc}.
\begin{figure}
\centering
    \begin{minipage}[t]{0.48\textwidth}
    \centering
   \includegraphics[width=\textwidth]{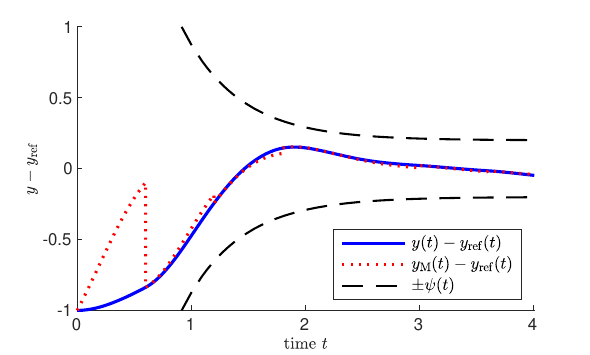}
\end{minipage}
\begin{minipage}[t]{0.48\textwidth}
\centering
    \includegraphics[width=\textwidth]{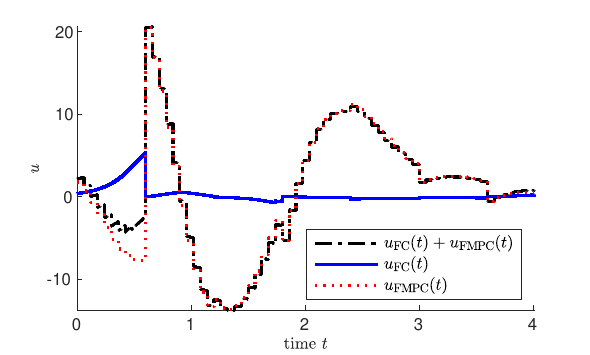}
\end{minipage}
    \caption{Simulation of system~\eqref{eq:ExampleMassOnCarSystem} under learning-based robust funnel MPC~\Cref{Algo:LearningRFMPC}.}
    \label{Fig:MassOnCar_learning_fmpc}
\end{figure}
It is evident that the control scheme is feasible and achieves the control objective.
Both errors~$y_{\rm M}-y_{\rf}$ and~$y-y_{\rf}$ evolve within the funnel boundaries given by~$\psi$. Overall, a similar behavior as in the first simulation (initial divergence, small funnel control signals afterwards, jumps at each learning step, etc.) can be observed.
We like to emphasize that already after the first learning step, the quality of the model is apparently  quite good such that the funnel controller only has to slightly compensate 
 for model errors and the control signal mainly consists of the control~$u_{\rm FMPC}$ generated by the model-based controller component.

\section{Conclusion}

We extended the recently developed robust funnel MPC algorithm~\cite{BergDenn23} by a learning component in order to improve the model-based control component by learning the model from system data.
Based on the characterization of the model class (\Cref{Def:Model}) and corresponding feasible learning structures (\Cref{Def:Learning}), we showed that the interplay between pure feedback control, model predictive control and learning is consistent and successful. 
In particular, the control objective~\eqref{eq:ControlObjective} of output reference tracking with prescribed performance is achieved.
Future research will focus on several particular aspects of the presented approach. 
Questions to be answered are, among other things, what it means to have a \emph{good} learning scheme; which existing techniques (fundamental lemma by Willems et al.~\cite{faulwasser2022behavioral},
Koopman framework~\cite{mauroy2020koopman}, Neural Networks, etc.) can be used to exploit the data; how can knowledge about the system be incorporated into the learning scheme; to name but three aspects.
Of special interest is the question of rigorously proving $\cM$-feasibility for more sophisticated learning algorithms.
Moreover, the extension of the presented results to systems of higher relative degree is subject of current research.
To this end, the results on robust funnel MPC presented in~\cite{BergDenn23} will be extended to higher relative degree. The approach presented here can then be transferred to systems of higher relative degree \textit{mutatis mutandis}.

\bibliographystyle{plain}
\small
\bibliography{\References}

\end{document}